\documentclass[10pt,reqno]{amsart}
\usepackage{graphicx,color}
\usepackage{amssymb,amscd,latexsym, mathrsfs, epsfig}
\usepackage[all]{xy}

\newcommand\C{\mathbb C}
\newcommand\Q{\mathbb Q}
\newcommand\R{\mathbb R}
\newcommand\Z{\mathbb Z}
\newcommand\N{\mathbb N}

\newcommand{\z}{\zeta}
\newcommand{\eps}{\varepsilon}

\newcommand\Aut{\operatorname{Aut}}

\newcommand\ms{\operatorname{MS}}

\newcommand\Ad{\operatorname{Ad}}

\newcommand{\bra}{\langle}
\newcommand{\ket}{\rangle}

\newcommand{\del}{\partial}

\newcommand{\B}{\mathcal{B}}
\newcommand{\G}{\mathcal{G}}
\newcommand{\M}{\mathcal{M}}
\newcommand{\T}{\mathcal{T}}
\newcommand{\W}{\mathcal{W}}

\newcommand{\X}{\mathcal{X}}
\newcommand{\Xsf}{\mathcal{X}^{\operatorname{sf}}}

\newcommand{\dt}{\operatorname{DT}}

\newcommand{\g}{\mathfrak{g}}

\makeatletter \@addtoreset{equation}{section} \makeatother

\newtheorem{thm}{Theorem}[section]
\newtheorem{lem}[thm]{Lemma}
\newtheorem{cor}[thm]{Corollary}

\newenvironment{rmk}{\noindent\textbf{Remark}.}{\\}
\newenvironment{exm}{\noindent\textbf{Example}.}{\\}

\title[]{TBA type equations and tropical curves}
\author{S. A. Filippini and J. Stoppa}
\date{}
\address{}
\email{}
\begin{document}
\begin{abstract} We revisit the wall-crossing behaviour of solutions to the Thermodynamic Bethe Ansatz type equations arising in a class of three-dimensional field theories, expressed as sums of ``instanton corrections". We explain how to attach to an instanton correction at a critical value a set of (combinatorial types of) tropical curves in $\R^2$ of fixed degree, which determines its jump to leading order. We show that a weighted sum over all such curves is in fact a tropical count. This goes through to the $q$-deformed setting. Our construction can be regarded as a formal mirror symmetric statement in the framework proposed by Gaiotto, Moore and Neitzke.
\end{abstract}

\maketitle

\section{Introduction} 

Motivated by physical arguments, Gaiotto, Moore and Neitzke \cite{gmn} proposed a new construction of holomorphic symplectic forms $\varpi(\z)$ (parametrized by $\z \in \C^*$) conjecturally yielding complete, smooth (or orbifold) hyperk\"ahler metrics on a class of torus fibrations $\pi\!: \M \to \B$. A mathematical introduction to this circle of ideas may be found in \cite{neitzke}. There has been considerable progress in relating this proposal to mirror-symmetric statements. In particular Chan \cite{chan} and Lu \cite{lu} interpreted the GMN construction in terms of mirror symmetry for $\M\to\B$ in the sense of Auroux \cite{auroux} or more generally of the Gross-Siebert program \cite{gross}. Some aspects of the works of Sutherland \cite{tomS}, Bridgeland-Smith \cite{tomB} and Kontsevich-Soibelman \cite{koso} are also related to this problem.   

The remarkable results mentioned above are global in nature. Here we describe a result that can be seen as a local, formal mirror-symmetric statement in the framework of \cite{gmn}. Suppose we restrict to a sufficiently small ball $B\subset \B$, such that the fibres of $\M_{B} \to B$ are smooth tori. Let $\Gamma$ denote the dual of the local system $\R^1\pi_*\Z$. Then the GMN holomorphic symplectic forms $\varpi(\z)$ on $\M_B$ are constructed using ``instanton corrections'', which are functions of $u \in B$ of the form
\begin{equation}\label{correction}
\G_{\T}(\z, u) = \int_{\ell(u)}\prod_{\{i\to j\}\subset \T} \frac{1}{2\pi i}\frac{d\z_j}{\z_j}\frac{\z_j + \z_i}{\z_j - \z_i} e^{-2\pi R(\z^{-1}_j Z_{\alpha(j)}(u) + \z_j \bar{Z}_{\alpha(j)}(u))}
\end{equation}
indexed by plane trees $\T = \{i_0 \to \T'\}$ for some nonempty, connected (naturally rooted) tree $\T'$ (setting $\z_{i_0} = \z$). Here $\alpha\!: \T'^{[0]} \to {\Gamma\setminus\{0\}}$ is a decoration of the vertices of $\T'$ by homology classes, $Z_{\bullet}(u)\!: \Gamma_u \to \C$ is a suitable homomorphism (a ``central charge", which in concrete applications is given by the periods of a meromorphic $1$-form), and $\ell(u)$ is the product of the rays $\ell_{\alpha(j)}(u) = \R_{<0} Z_{\alpha(j)}(u) \subset \C^*$. By construction $\G_{\T}(u)$ (for fixed, general $\z$) is a smooth function of $u$ away from the locus in $B$ where two or more rays $\ell_{\alpha}$ coincide; this is usually denoted by $\ms \subset B$ (the ``wall of marginal stability").

We regard the functions $\G_{\T}(\z, u)$ as B-model objects, since they encode corrections to the complex structure on $\M$ which makes $\varpi(\z)$ holomorphic. As $u$ crosses $\ms$, the instanton corrections $\G_{\T}(\z, u)$ change in a discontinuous way; this should correspond (conjecturally) to the wall-crossing of holomorphic discs with boundary on a Special Lagrangian fibre $\M_{u}$. At any rate, the corrections $\G_{\T}(\z, u)$ depend (by definition) on the choice of a fibre, and one would expect some interesting geometric information to emerge when $u$ crosses $\ms$. This is precisely what we wish to test, at a formal level, in this paper.

In the following we assume that $\alpha\!: \T'^{[0]} \to {\Gamma\setminus\{0\}}$ takes values in the monoid $\Gamma^+_{\gamma, \eta}$ generated by two classes $\gamma, \eta$ (we denote by $\Gamma_{\gamma, \eta}$ the corresponding lattice). We denote by $\deg(\T)$ the unordered collection $\{\alpha(i), i \in \T'^{[0]}\}$, and by $c(\T)$ the corresponding sum $\sum_{i \in \T'^{[0]}} \alpha(i)$.  By a labelling $\nu$ of $\T$ we mean a total order of $\T'^{[0]}$; we do not assume that $\nu$ is compatible with the natural orientation of $\T'$. It will be sufficient to consider trees for which the decoration $\gamma_{\T}$ attached to the root of $\T'$ is a (positive) multiple of $\gamma$.

A crucial role is played by a class of rational tropical curves $\Upsilon$ in $\R^2 = \Gamma_{\gamma, \eta}\otimes\R$, studied in \cite{gps} Section 2. Fixing a tropical degree ${\bf w}$ and a general collection of infinite ends $\mathfrak{d}$ with directions $-\gamma, -\eta$, weighted by ${\bf w}$, we denote by $\mathcal{S}(\mathfrak{d}, {\bf w})$ the set of connected, rational tropical curves spanning $\mathfrak{d}$, with a single additional end, and by $N^{\rm trop}({\bf w})$ the associated tropical count. We denote the combinatorial type of such a curve by $[\Upsilon]$. We will also write $[\mathcal{S}(\mathfrak{d}, {\bf w})]$ for the collection of combinatorial types occurring in $\mathcal{S}(\mathfrak{d}, {\bf w})$. 

We regard the tropical curves $\Upsilon$ above as A-model objects. On the one hand, the tropical counts $N^{\rm trop}({\bf w})$ encode certain relative Gromov-Witten invariants in weighted projective planes (\cite{gps} Theorems 3.4 and 4.4). On the other, they are conjecturally related to counts of holomorphic discs (see \cite{gross} Section 11). 

We will use the wall-crossing of a B-model object $\G_{\T}(\z)$ to construct A-model objects, namely a set of (combinatorial types of) tropical curves $[\Upsilon_j]$, together with signs $\pm 1$. Then we will check that the total count on the A-model (weighted by some natural factors), over all trees whose degree $\deg(\T)$ is identified with a tropical degree ${\bf w}$, is in fact the tropical invariant $N^{\rm trop}({\bf w})$. The proof relies on the technique of \cite{gps} Theorem 2.8. Thus the present work may be seen as an application of the methods developed in \cite{gps} to the study of the instanton corrections \eqref{correction}.

Fix a smooth point $u_0 \in \ms$. Suppose that $p(\tau)\!: [0, 1]\to B$ is a smooth path which only intersects $\ms$ transversely in $u_0$ for $\tau = \tau_0$. Let $\mathcal{C}$ be the union of the cones spanned in $\C$ by $Z_{\pm\gamma}(p(0)), Z_{\pm\eta}(p(0))$. Assume that the slopes of $Z_{\gamma}(p(0)), Z_{\eta}(p(0))$ and $Z_{\gamma}(p(1)), Z_{\eta}(p(1))$ are interchanged (in the positive quadrant, with $Z_{\gamma}(p(0))$ preceding $Z_{\eta}(p(0))$ in the clockwise order, and $Z_{\gamma}(p(1)), Z_{\eta}(p(1))$ contained in $\mathcal{C}$). Let
\begin{equation*}
J_{\T}(\tau^{\pm}, R) = \G_{\T}(p(\tau^-)) - \G_{\T}(p(\tau^+)),\quad \tau^- < \tau_0 < \tau^+.
\end{equation*}
\noindent\textbf{Lemma A.} (Lemma \ref{LemmaA}). \emph{Fix a tree $\T$ and a labelling $\nu$ as above. Let $\z \in \C\setminus\mathcal{C}$, and suppose that $c(\T) \in \Gamma$ is primitive.}
\begin{enumerate}
\item \emph{There is a natural construction (depending on $\nu$) that associates to an instanton correction $\G_{\T}(\z)$ and a critical value $u_0$ a set $\{[\Upsilon_j]\}_{\T}$ of combinatorial types of plane rational tropical curves of degree $\deg(\T)$, with signs $(-1)^{[\Upsilon_j]} \in \{\pm 1\}$. For suitable $\nu$ these combinatorial types (as $\T$ varies) all belong to a single set $[\mathcal{S}(\mathfrak{d}, \deg(\T))]$.}

\item \emph{$J_{\T}(\tau^{\pm}, R)$ has an asymptotic expansion, as $R \to +\infty$,
\begin{align*}
J_{\T}(\tau^{\pm}, R) = &\sum_{\{[\Upsilon_j]\}_{\T}} (-1)^{[\Upsilon_j]} f(\z, R)\\ &+ o_{\tau^{\pm}}(\frac{1}{\pi |Z_{c(\T)}(u_0)| R} e^{-\pi |Z_{c(\T)}(u_0)| R}) + O(|\tau^+ - \tau^-|),  
\end{align*}
where the $o_{\tau^{\pm}}$ error term is for fixed $\tau^{\pm}$, sufficiently close to $\tau_0$, the $O(|\tau^+ - \tau^-|)$ term is uniform in $R \geq 1$ as $|\tau^+ - \tau^-| \to 0$, and we have, nontangentially to $\partial\mathcal{C}$,} 
\begin{equation*}
\lim_{\z \to 0} f(\z, R) \sim \frac{1}{\pi |Z_{c(\T)}(u_0)| R}e^{-\pi |Z_{c(\T)}(u_0)| R},
\end{equation*}
\emph{where the expansion is for $R \to +\infty$ and holds uniformly for $|\tau^+ - \tau^-|$ sufficiently small.}
\end{enumerate}
Let us now sum over all trees with the same $\deg(\T)$, identified with a tropical degree ${\bf w} = ({\bf w}_{1}, {\bf w}_2) = (w_{i j})$, and with a fixed root label (say $w_{11}\gamma$), 
\begin{equation*}
J_{{\bf w}}(\tau^{\pm}, R) = \sum_{\deg(\T) = {\bf w}, \gamma_{\T} = w_{11}\gamma} \W_{\T} J_{\T}(\tau^{\pm}, R),
\end{equation*}
(the weights $\W_{\T}$, which are natural in the framework of \cite{gmn}, will be introduced in section \ref{TBAsection}). Notice that $c(\T)$ is fixed and can be identified with the integral vector $(|{\bf w}_1|, |{\bf w}_2|)$. We write ${\bf w}'$ for the degree obtained by dropping $w_{11}$.\\

\noindent\textbf{Theorem B.} (Theorem \ref{TheoremB}). \emph{Let $(|{\bf w }_1|, |{\bf w}_2|)$ be primitive. There is an expansion, as $R \to \infty$,}
\begin{align*}
J_{{\bf w}}(\tau^{\pm}, R) & = \prod_{i, j}\frac{1}{w^2_{i j}}\frac{N^{\rm trop}({\bf w})}{\Aut({\bf w}')} f(\z, R)\\ & + o_{\tau^{\pm}}(\frac{1}{\pi |Z_{c(\T)}(u_0)| R} e^{-\pi |Z_{c(\T)}(u_0)| R}) + O(|\tau^+ - \tau^-|).
\end{align*}
\emph{Equivalently,}
\begin{equation*}
\sum_{\deg{\T} = {\bf w}, \gamma_{\T} = w_{11}\gamma} \W_{\T} \left(\sum_{\{[\Upsilon_j]\}_{\T}} (-1)^{[\Upsilon_j]}\right) = \prod_{i, j}\frac{1}{w^2_{i j}}\frac{N^{\rm trop}({\bf w})}{\Aut({\bf w}')}.
\end{equation*}\\
 
\begin{rmk} The assumption that $c(\T)$ is primitive in Lemma \ref{LemmaA} (and similarly that $(|{\bf w }_1|, |{\bf w}_2|)$ is primitive in Theorem \ref{TheoremB}) is made mostly for simplicity of exposition. We can prove almost the same statements in the general case, with one important difference: we have not been able to show that for suitable $\nu$ the tropical types $[\Upsilon_j]$ are connected.
\end{rmk}

\begin{exm} Consider the integral corresponding to $\T = \{ i_0 \to i \to j \}$ with $\alpha(i) = \gamma$, $\alpha(j) = \eta$, at a fixed point $u^+ = p(\tau^-)$ for $\tau^- < \tau_0$, sufficiently close to $\tau_0$,
\begin{equation}\label{IntroExaInt1}
\int_{\ell^+_{\gamma}}\frac{1}{2\pi i}\frac{d\z_i}{\z_i}\frac{\z_i + \z}{\z_i - \z} e^{-2\pi R(\z^{-1}_i Z^+_{\gamma} + \z_i \bar{Z}^+_{\gamma})}\int_{\ell^+_{\eta}}\frac{1}{2\pi i}\frac{d\z_j}{\z_j}\frac{\z_j + \z_i
}{\z_j - \z_i} e^{-2\pi R(\z^{-1}_j Z^+_{\eta} + \z_j \bar{Z}^+_{\eta})}.  
\end{equation}
(denoting with a $+$ label quantities evaluated at $u^+$). Choose $u^- = p(\tau^+)$ for $\tau^+ = 2\tau_0 - \tau^-$. By the residue theorem (and denoting with a $-$ label quantities evaluated at $u^-$), for general $\z$ we can rewrite \eqref{IntroExaInt1} as
\begin{equation}\label{IntroExaInt2}
\int_{\ell^-_{\gamma}}\frac{1}{2\pi i}\frac{d\z_i}{\z_i}\frac{\z_i + \z}{\z_i - \z} e^{-2\pi R(\z^{-1}_i Z^+_{\gamma} + \z_i \bar{Z}^+_{\gamma})}\int_{\ell^-_{\eta}}\frac{1}{2\pi i}\frac{d\z_j}{\z_j}\frac{\z_j + \z_i
}{\z_j - \z_i} e^{-2\pi R(\z^{-1}_j Z^+_{\eta} + \z_j \bar{Z}^+_{\eta})}
\end{equation}
plus a residue term
\begin{equation}\label{IntroExaInt3}
\int_{\ell^-_{\gamma + \eta}}\frac{1}{2\pi i}\frac{d\z_i}{\z_i}\frac{\z_i + \z}{\z_i - \z} e^{-2\pi R(\z^{-1}_i Z^+_{\gamma + \eta} + \z_i \bar{Z}^+_{\gamma + \eta})}.  
\end{equation}
We regard this computation (attaching to \eqref{IntroExaInt1} the residue \eqref{IntroExaInt3}) as defining the combinatorial type of a tropical line in $\R^2 = \Gamma_{\gamma, \eta}\otimes\R$, with infinite ends in the direction of $-\gamma, -\eta$ and $\gamma+\eta$. This is the unique element of $[\mathcal{S}(1,1)]$. 
\begin{figure}[ht]
\centerline{\includegraphics[scale=.8]{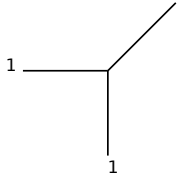}}
\caption{Tropical type of \eqref{IntroExaInt1}.}
\end{figure}
Since $Z(p(\tau))$ is continuous across $\tau_0$ the limit of \eqref{IntroExaInt2} as $\tau^+ \to \tau^+_0$ equals the limit as $\tau^- \to \tau^-_0$ of
\begin{equation*}
\int_{\ell^-_{\gamma}}\frac{1}{2\pi i}\frac{d\z_i}{\z_i}\frac{\z_i + \z}{\z_i - \z} e^{-2\pi R(\z^{-1}_i Z^-_{\gamma} + \z_i \bar{Z}^-_{\gamma})}\int_{\ell^-_{\eta}}\frac{1}{2\pi i}\frac{d\z_j}{\z_j}\frac{\z_j + \z_i
}{\z_j - \z_i} e^{-2\pi R(\z^{-1}_j Z^-_{\eta} + \z_j \bar{Z}^-_{\eta})}, 
\end{equation*} 
and similarly the limit of \eqref{IntroExaInt3} as $\tau^+ \to \tau^+_0$ equals 
\begin{equation*}
\int_{\ell^0_{\gamma + \eta}}\frac{1}{2\pi i}\frac{d\z_i}{\z_i}\frac{\z_i + \z}{\z_i - \z} e^{-2\pi R(\z^{-1}_i Z^0_{\gamma + \eta} + \z_i \bar{Z}^0_{\gamma + \eta})}.  
\end{equation*}
Thus the jump of $\G_{\T}(p(\tau))$ across $\tau_0$ is given by the latter integral, which when $\z \to 0$ nontangentially to $\ell^0_{\gamma + \eta}$ looks like $\frac{1}{\pi |Z^0_{\gamma + \eta}| R}e^{-\pi  |Z^0_{\gamma + \eta}| R} + o(\frac{1}{\pi |Z^0_{\gamma + \eta}| R} e^{-\pi |Z^0_{\gamma + \eta}|R})$. From this we read off $N^{\rm trop}(1,1) = 1$. 
\end{exm}

\begin{rmk} From the point of view of \cite{gmn} it is natural to consider corrections $\G_{\T}(\z)$ which involve $m\geq 3$ or more lattice elements $\{\gamma_i\}$, with $Z(\gamma_i)$ spanning distinct rays. According to the construction in Lemma \ref{LemmaA}, these appear to be naturally related to tropical curves in $\R^m$. It might be interesting to see if these also have an enumerative meaning as in Theorem \ref{TheoremB}.
\end{rmk}\\

\noindent In the rest of the paper we will forget the fibration $\M\to\B$ (except for a few comments), and work formally with the instanton corrections \eqref{correction}. In \cite{gmn}, these corrections arise from iterative solutions of a Thermodynamic Bethe Ansatz type integral equation. In section \ref{TBAsection} we introduce a more algebraic version of this equation, for which all convergence problems become trivial, and define instanton corrections in this context. As we will explain, we are effectively replacing the complicated diffeomorphism group $\operatorname{Diff}(\M_u, \Gamma\otimes_{\Z}\C^*)$ which appears in \cite{gmn} with the automorphism group $\Aut_R(\widehat{\g})$ of a suitable algebra over an Artin ring. In section \ref{tropicalSection}, after briefly recalling the few notions we need from tropical geometry, we use some basic results about this formal TBA equation to prove Lemma A (Lemma \ref{LemmaA}) and Theorem B (Theorem \ref{TheoremB}). Finally in section \ref{qSection} we introduce a natural $q$-deformation of our formal TBA equation, and briefly explain how Lemma A and Theorem B go through to the $q$-deformed setting, replacing the usual tropical counts with suitable $q$-deformed ones (Block-G\"ottsche invariants).\\

\noindent\textbf{Acknowledgements.} We are very grateful to Tom Bridgeland and Mario Garcia Fernandez for some conversations related to the material presented here. This research was supported in part by the Hausdorff Institute for Mathematics, Bonn (JHRT ÒMathematical PhysicsÓ), GRIFGA 2012-2015 and ERC Starting Grant 307119.

\section{The TBA type equation}\label{TBAsection}

Let $\Gamma \cong \Z^{2r}$ be a fixed lattice with a skew-symmetric, bilinear form $\bra -, - \ket$. A \emph{central charge} is a homomorphism $Z\!: \Gamma \to \C$ (once a fixed basis of $\Gamma$ is declared to be positive, one usually requires that elements of the positive subsemigroup $\Gamma^+$ are mapped to the upper half plane $\mathbb{H}$). The choice of $Z$ induces a notion of \emph{slope} for the elements $\alpha \in \Gamma$, which is simply the slope of the ray containing the image $Z(\alpha)$ (if $Z(\alpha) = 0$ the slope is undefined).\\

\noindent{\textbf{Assumption.}} In the rest of this paper we will assume that lattices $(\Gamma, \bra - , - \ket)$ and central charges $Z$ satisfy the condition that for elements $\alpha, \beta$ having the same slope (with respect to $Z$) or with $Z(\alpha) = Z(\beta) = 0$, one has $\bra \alpha, \beta \ket = 0$. This is certainly true if $\Gamma \cong \Z^2$ and $Z$ is nondegenerate (i.e. the induced linear map $\Gamma\otimes_{\Z}\R \to \R^2$ has maximal rank), but it is enough to require that $\Gamma$ splits as $\Gamma_1 \oplus \Gamma_2$ with $\Gamma_1 \cong \Z^2, \bra \Gamma_2, \Gamma \ket = 0$ and that $Z$ factors through the projection $\pi\!: \Gamma \to \Gamma_1$, giving a nondegenerate map $\Gamma_1 \to \C$. We will always denote by $\pi$ this projection. In the wall-crossing literature one usually restricts to the lattice generated by two fixed elements. Our assumption gives a very similar restriction, but for our purposes we will need to be able to work with linearly independent $\gamma_i \in \Gamma$ with $Z(\gamma_i) = Z(\gamma_j)$ and $\bra \gamma_i, \gamma_j\ket = 0$.\\

\noindent Consider the infinite dimensional complex Lie algebra $\g$ generated by $e_{\alpha}, \alpha \in \Gamma$ with bracket 
\begin{equation}\label{classBracket}
[e_{\alpha}, e_{\beta}] = (-1)^{\bra \alpha, \beta\ket}\bra \alpha, \beta\ket e_{\alpha + \beta}.
\end{equation}
We also endow $\g$ with the associative, commutative product determined by 
\begin{equation}\label{classProduct}
e_{\alpha} e_{\beta} = (-1)^{\bra \alpha, \beta\ket} e_{\alpha + \beta}.
\end{equation}
With the associative product \eqref{classProduct} and bracket \eqref{classBracket}, $\g$ becomes a Poisson algebra: the linear map $[x, -]$ satisfies the Leibniz rule. Notice that by our assumption elements $e_{\alpha}, e_{\beta}$ of $\g$ such that $\alpha, \beta$ have the same slope (with respect to $Z$) Poisson commute: $[e_{\alpha}, e_{\beta}] = 0$.\\ 

\noindent A \emph{BPS spectrum} (for the fixed central charge $Z$) is a function $\Omega\!: \Gamma \to \mathbb{Z}$ such that $\Omega(\gamma) = \Omega(-\gamma)$ and $\Omega(0) = 0$. The \emph{BPS rays} of $\Omega$ for a fixed central charge are the rays $\ell_{\alpha} \subset \C^*$ of the form $\R_{< 0}Z(\alpha)$ where $\Omega(\alpha) \neq 0$.\\

\noindent Let $R$ be an Artin $\C$-algebra or a complete local $\C$-algebra, with maximal ideal $\mathfrak{m}_R\subset R$. We write $\widehat{\g}$ for the completed tensor product of $\g$ with $R$,
\begin{equation*}
\widehat{\g} = \g\,\widehat{\otimes}_{\C} R = \lim_{\longleftarrow} \g \otimes_{\C} R/{\mathfrak{m}^k_R}.
\end{equation*}  

We say that a family of automorphisms $\psi(\z) \in \Aut_R(\widehat{\g})$ (automorphisms of $\widehat{\g}$ as an associative, commutative algebra) parametrised by an open set $U \subset \C^*$ is \emph{holomorphic} if for all $\alpha \in \Gamma$ the image $\psi(\z)(e_{\alpha})$ is a holomorphic function on $U$ with values in a linear subspace of $\widehat{\g}$ generated by finitely many $\{r e_{\beta}\}, r \in R$ (in other words there is a Laurent expansion for $\psi(\z)(e_{\alpha})$ such that the $\widehat{\g}$-valued coefficients involve only a finite number of vector space generators $\{r e_{\beta}\}$). Moreover if $\del U$ contains the BPS ray $\ell_{\alpha}$, we require $\psi(\z)(e_{\alpha})$ to extend to a holomorphic function of the same kind in a neighborhood of $\ell_{\alpha}$ in $\C^*$. We denote the space of all such families by $\mathcal{H}(U, \Aut_R(\widehat{\g}))$. From now we write $\psi_{\alpha}(\z)$ or simply $\psi_{\alpha}$ for $\psi(\z)(e_{\alpha})$. If we replace $\z$ with a real variable $\tau \in I$ (for $I \subset \R$ an interval) we may give a similar definition for a \emph{continuous} family $\psi(\tau)$; in other words we now require that $\psi(\tau)(e_{\alpha})$ is a linear combination of finitely many $\{r e_{\beta}\}$, with complex valued continuous coefficients. We denote this set by $C^0(I, \Aut_R(\widehat{\g}))$.

Suppose now that $\Gamma$ is generated by elements ${\gamma_1, \dots, \gamma_{\ell_1}}$ and $\eta_1, \dots, \eta_{\ell_2}$ such that $\bra \gamma_i, \gamma_j\ket = \bra \eta_i, \eta_j \ket = 0, \bra \gamma_i, \eta_j\ket = 1$. Let all the $\gamma_i$ ($\eta_j$) have a common central charge $Z(\gamma_i)$ (respectively $Z(\eta_j)$). We choose $\pi\!:\Gamma \to \Z^2$ given by $\pi(\gamma_i) = (0, 1), \pi(\eta_j) = (1, 0)$.

Let $\widehat{\g}$ be the completion over the base ring 
\begin{equation*}
R_k = \frac{\C[[s_1, \dots, s_{\ell_1}, t_1, \dots, t_{\ell_2}]]}{(s^{k+1}_1, \dots, s^{k+1}_{\ell_1}, t^{k+1}_1, \dots, t^{k+1}_{\ell_2})}.
\end{equation*}
Choose $U$ to be the complement in $\C^*$ of finitely many BPS rays $\{\ell_{\gamma'}\}$: we remove those for which $\gamma' = \sum^{\ell_1}_{i = 1} a_i \gamma_i + \sum^{\ell_2}_{j = 1} b_j \eta_j$ with $|a_i|, |b_j| \leq k$. We write $(s, t)^{\gamma'}$ for the product $\prod_i s^{|a_i |}_i \prod_j t^{|b_j |}_j$. Fix a reference $f \in \mathcal{H}(\C^*, \Aut_{R_k}(\widehat{\g}))$ (notice that $f$ is required to be holomorphic on all $\C^*$). We introduce a kernel
\begin{equation*}
\rho(\z, \z') = \frac{1}{4\pi i}\frac{\z' + \z}{\z' - \z}.
\end{equation*}
We wish to define an integral operator $\Phi$ on (a suitable subspace of) $\mathcal{H}(U, \Aut_{R_k}(\widehat{\g}))$ as
\begin{equation}\label{TBop}
(\Phi(\psi))_{\alpha}(\z) = f_{\alpha}(\z)\exp\left(\sum_{\gamma'} \Omega(\gamma')\bra \gamma', \alpha\ket \int_{\ell_{\gamma'}}\frac{d\z'}{\z'}\rho(\z, \z')\log(1 - (s, t)^{\gamma'}\psi_{\gamma'}(\z'))\right),
\end{equation}
summing over all $\gamma' \in \Gamma$ with $\Omega(\gamma') \neq 0$. 

\begin{rmk} The exponential here might be a little confusing: what we mean is the exponential of a \emph{commutative} power series (i.e. using the commutative algebra structure \eqref{classProduct} on $\widehat{\g}$).
\end{rmk}

\noindent In fact since we are working over $R_k$, $(s, t)^{\gamma'}$ vanishes for all but finitely many $\gamma'$, and the power series expansions of $\log(1 - (s, t)^{\gamma'}\psi_{\gamma'}(\z'))$ are all finite. Therefore the integrals on the right hand side of \eqref{TBop} make sense as line integrals of a holomorphic function with values in a finite dimensional vector space, and the coefficients of the Laurent expansion for $(\Phi(\psi))_{\alpha}(\z)$ involve finitely many $\{e_{\beta}\}$. Also (at least formally, modulo convergence of each integral) we have $\Phi(\psi)_{\alpha + \beta} = (-1)^{\bra \alpha, \beta\ket}\Phi(\psi)_{\alpha} \Phi(\psi)_{\beta}$. In other words $\Phi$ (when well defined) actually gives an endomorphism of $\widehat{\g}$.\\

\noindent\begin{rmk} For a suitable choice of $f$ (to be specified in a moment) $\Phi$ is formally the same integral operator appearing in \cite{gmn} Section 5.3 (see also \cite{neitzke} section 4.2). In that context solutions of $\X = \Phi(\X)$ are denoted by $\X_{\alpha}(m, R; \z)$, where $m \in \M$ is a point in the moduli space of the underlying physical theory compactified on a circle $S^1_R$ of radius $R$. The $\X_{\alpha}(m, R; \z)$ have a geometric meaning as (exponential, holomorphic) Darboux coordinates for the  hyperk\"ahler metric on the moduli space. Then as explained in ibid. Section 5 the integral operator has a beautiful interpretation as the superposition of the operators whose fixed points are the coordinates for the exact $1$-particle hyperk\"ahler metrics. The $\X_{\alpha}(m, R; \z)$ can also be regarded collectively as a family of diffeomorphisms $\X(u, R; \z)\!: \M_u \to \Gamma\otimes_{\Z}\C^*$ which is holomorphic in $\z$, where $u \in \mathcal{B}$ is a point in the Coulomb branch of the theory (so the fibre $\M_u$ is a compact torus, locally $\Gamma\otimes_{\Z} \R / 2\pi\Z$). In our definition of $\Phi$ we have replaced $\operatorname{Diff}(\M_u, \Gamma\otimes_{\Z}\C^*)$ with $\Aut_{R_k}(\widehat{\g})$.
\end{rmk}

\noindent Fix $R > 0$. The \emph{semiflat automorphism of} $\widehat{\g}$ \emph{of radius} $R$ (with respect to a fixed central charge $Z$) is the element of $\Aut_{R_k}(\widehat{\g})$ given by
\begin{equation*}
\psi^{0}_{\alpha}(\z) = \exp(\pi R (\z^{-1} Z(\alpha) + \z\bar{Z}(\alpha))) e_{\alpha}.
\end{equation*}
We introduce a subspace $\widetilde{\mathcal{H}} \subset \mathcal{H}(U, \Aut_{R_k}(\widehat{\g}))$ spanned by those functions for which the limit
\begin{equation*}
\lim_{\z \to 0} \exp(-\pi R (\z^{-1} Z(\alpha) + \z\bar{Z}(\alpha))) \psi_{\alpha}(\z)
\end{equation*} 
exists in $\widehat{\g}$ for all $\alpha$, and the same holds for the limit as $\z \to \infty$, except possibly when $\z \to 0$  or $\z \to \infty$ tangentially to a BPS ray.\\

\begin{rmk} In our present, simplified context the $\psi^{0}_{\alpha}(\z)$ replace the semiflat coordinates $\Xsf_{\alpha}(m, R; \z)$ of \cite{gmn} Section 3.3 (obtained by ``naive dimensional reduction").
\end{rmk} 

\noindent \noindent From now on we pick $f = \psi^0$ in the definition of $\Phi$.
\begin{lem} The operator $\Phi$ given by \eqref{TBop} (with the choice $f = \psi^0$) is well defined $\widetilde{\mathcal{H}} \to \widetilde{\mathcal{H}}$.
\end{lem}
\begin{proof} First we show that for fixed $\psi \in \widetilde{\mathcal{H}}$, $\alpha \in \Gamma$ and $\zeta \in U$ the right hand side of \eqref{TBop} is a well defined element of $\widehat{\g}$. Since $\z$ lies away from the rays of integration, the function $(\z')^{-1}\rho(\z, \z')$ is holomorphic on each $\ell_{\gamma'}$. And as $(s, t)^{\gamma'} = 0$ in $R_k$ for all but finitely many $\gamma'$ we are effectively summing over a finite number of $\gamma'$. Moreover all the integrals appearing are convergent. To see this we first expand $\log(1 - (s, t)^{\gamma'}\psi_{\gamma'}(\z'))$ in $R_k$, then use the boundary conditions in $\widetilde{\mathcal{H}}$ for $\z \to 0$, $\z \to \infty$ to see that each integral is dominated by the sum of a fixed finite number of integrals of the form
\begin{equation*}
C \int_{\ell_{\gamma'}}\frac{d\z'}{\z'}\frac{\z' + \z}{\z' - \z}\exp(\pi R (\z'^{-1} Z(\gamma') + \z'\bar{Z}(\gamma'))).
\end{equation*}  
for some constant $C$. These are all convergent; we will study these and more general integrals in Lemma \ref{BesselLemma} below. Then it follows from standard theory that the right hand side of \eqref{TBop} is a holomorphic function of $\z \in U$ (since $(\z')^{-1}\rho(\z, \z')$ is, and by the above convergence). On the other hand it is also holomorphic in a neighborhood of $\ell_{\alpha}$, since the integral along $\ell_{\alpha}$ appears with a factor of $\bra \alpha, \alpha\ket$ and therefore vanishes. Finally we can use the Lemma \ref{BesselLemma} to take the limit of 
\begin{equation*}
\exp(-\pi R (\z^{-1} Z(\alpha) + \z\bar{Z}(\alpha))) \Phi(\psi)_{\alpha}(\z)
\end{equation*}
as $\z \to 0$ (at least nontangentially to a BPS ray): by the definition of $\psi^0$, this is a constant element of $\widehat{\g}$ given by 
\begin{equation*}
\exp\left(\sum_{\gamma'} \Omega(\gamma')\bra \gamma', \alpha\ket\frac{1}{4\pi i}\int_{\ell_{\gamma'}}\frac{d\z'}{\z'}\log(1 - (s, t)^{\gamma'}\psi_{\gamma'}(\z'))\right)e_{\alpha}.
\end{equation*}
The same argument applies to the $\z \to \infty$ limit, which completes the check that $\Phi(\psi)$ lies in $\widetilde{\mathcal{H}}$.  
\end{proof}
\noindent The basic integral we need to consider has the form
\begin{equation}\label{basicIntegral}
\int_{\R_{< 0} e^{i\psi} c} \frac{d\z'}{\z'}\frac{\z'+\z}{\z'-\z} \exp(\pi R(\z'^{-1}c + \z' \bar{c}))
\end{equation}
where $c\in \C^*$, $\psi$ is a sufficiently small angle $|\psi| < \eps$, and $\z \notin \R_{< 0} e^{i\psi} c$. For later applications we also look at the integral along an arc,
\begin{equation}\label{arcIntegral}
\int_{|\psi|<\eps} \frac{d\z'}{\z'}\frac{\z'+\z}{\z'-\z} \exp(\pi R (\z'^{-1} c + \z' \bar{c}))
\end{equation}
where $\z' \in \R_{< 0} e^{i\psi} c$, $|\z'|$ is fixed, $\z \notin \R_{< 0} c$ and $\eps$ is small enough.  
\begin{lem}\label{BesselLemma} The integral \eqref{basicIntegral} converges for sufficiently small $|\psi|$, and in fact its modulus is bounded above by $\frac{C}{2\pi R |c|}\exp(-2\pi R |c|)$ for $R$ large enough (for a constant $C$ depending on the angular distance of $\z$ from $\R_{< 0} e^{i\psi} c$). Similarly for $\eps$ sufficiently small \eqref{arcIntegral} vanishes as $|\z'| \to 0$ or $|\z'| \to \infty$. 
\end{lem}
\begin{proof}
We make the change of variable $\z' = -e^{s + i \psi} c$ for real $s$ and $\psi$, reducing \eqref{basicIntegral} to
\begin{equation*}
\int^{+\infty}_{-\infty} ds \frac{-e^{s + i\psi} c + \z}{-e^{s + i\psi} c - \z}\exp(-\pi R|c|(e^{-s - \log|c| - i\psi} + e^{s + \log|c| + i\psi}))
\end{equation*}
The modulus of this is bounded above by 
\begin{equation*}
C \int^{+\infty}_{-\infty} ds |\exp(-2\pi R |c| \cosh(s + \log|c|+ i\psi))|,
\end{equation*} 
where $C$ is a constant depending on the angular distance of $\z$ from $\R_{< 0} e^{i\psi} c$. For $\eps$ sufficiently small this is in turn bounded by 
\begin{equation*}
C \int^{+\infty}_{-\infty} ds' \exp(-2\pi R |c| \cosh(s')),
\end{equation*}
where $C$ is a new (possibly larger) constant, depending also on $\eps$. We recognize the integral as a Bessel function, and we find that, for $\eps \ll 1$, \eqref{basicIntegral} converges for $|\psi| < \eps$, and moreover it is actually bounded above by 
\begin{equation*}
\frac{C}{2\pi R |c|}\exp(-2\pi R |c|)
\end{equation*}
for $R$ large enough. On the other hand, with the usual change of variable \eqref{arcIntegral} becomes
\begin{equation*}
i \int^{+\eps}_{-\eps} d\psi \frac{-e^{s + i\psi}c + \z}{-e^{s + i\psi}c - \z}\exp(-\pi R|c| e^{-s - \log|c| - i\psi})\exp(-\pi R|c| e^{s + \log|c| + i\psi}). 
\end{equation*}
One can check that this vanishes for $s \to \pm \infty$ (for fixed, sufficiently small $\epsilon$).
\end{proof}
\noindent We will need to know the behaviour of the function $\Phi(\psi)_{\alpha}(\z)$ (which is holomorphic in $U$) when $\z$ crosses a BPS ray $\ell_{\gamma'}$. 
\begin{lem}\label{JumpLemma} Let $\psi \in \widetilde{\mathcal{H}}$. Fix a ray $\ell$ and $\z_0 \in \ell$, and denote by $\Phi(\psi)_{\alpha}(\z^-_0)$ the limit as $\z \to \z_0$ in the counterclockwise direction. Similarly let $\Phi(\psi)_{\alpha}(\z^+_0)$ denote the limit in the clockwise direction. Both limits $\Phi(\psi)_{\alpha}(\z^{\pm}_0)$ exist, and they are related by
\begin{equation*} 
\Phi(\psi)_{\alpha}(\z^+_0) = \Phi(\psi)_{\alpha}(\z^-_0)\prod_{Z(\gamma')\in \ell}(1 - (s, t)^{\gamma'}\psi_{\gamma'}(\z_0))^{\Omega(\gamma') \bra \gamma', \alpha\ket}. 
\end{equation*}
\end{lem}
\begin{proof} Consider the integral
\begin{equation*}
\int_{\ell_{\gamma'}}\frac{d\z'}{\z'}\frac{\z' + \z}{\z' - \z}\log(1 - (s, t)^{\gamma'}\psi_{\gamma'}(\z')) = \sum^{k}_{p = 1}\frac{1}{p}(s, t)^{p\gamma'}\int_{\ell_{\gamma'}}d\z'\frac{1}{\z' - \z}\left(1 + \frac{\z}{\z'}\right)\psi_{p\gamma'}(\z')
\end{equation*}
($(s, t)^{p\gamma'} = 0$ in $R_k$ for some $p \leq k$). By estimates similar to those in the proof of Lemma \ref{BesselLemma} we may apply Plemelj's theorem to find
\begin{equation*}
\lim_{\z\to\z^{\pm}_0} \frac{1}{2\pi i}\int_{\ell_{\gamma'}}d\z'\frac{1}{\z' - \z}\left(1 + \frac{\z}{\z'}\right)\psi_{p\gamma'}(\z') =  \pm \psi_{p\gamma'}(\z_0) + \operatorname{pv}\frac{1}{2\pi i}\int_{\ell_{\gamma'}}\frac{d\z'}{\z'}\frac{\z' + \z_0}{\z' - \z_0}\psi_{p\gamma'}(\z'),
\end{equation*}
where $\operatorname{pv}$ denotes a (well defined, convergent) principal value integral. Therefore
\begin{align*}
\lim_{\z\to\z^{\pm}_0}\frac{1}{4\pi i}\int_{\ell_{\gamma'}}\frac{d\z'}{\z'}\frac{\z' + \z}{\z' - \z}\log(1 - &(s, t)^{\gamma'}\psi_{\gamma'}(\z')) = \pm \frac{1}{2}\log(1 - (s, t)^{\gamma'}\psi_{\gamma'}(\z_0))\\ 
&+ \operatorname{pv}\frac{1}{4\pi i}\int_{\ell_{\gamma'}}\frac{d\z'}{\z'}\frac{\z' + \z_0}{\z' - \z_0}\log(1 - (s, t)^{\gamma'}\psi_{\gamma'}(\z')),
\end{align*}
where the last principal value integral is convergent. The result now follows easily.
\end{proof}

\noindent Fixed points of $\Phi$ (acting on $\widetilde{\mathcal{H}}$) are solutions of the \emph{TBA type equation} 
\begin{equation}\label{TBA}
\psi = \Phi(\psi). 
\end{equation}
(see \cite{gmn} equation 5.13). 

We recall that solutions of \eqref{TBA} are in fact solutions of a Riemann-Hilbert factorization problem for the group $\Aut_R(\widehat{\g})$. Consider the standard dilogarithm
\begin{equation*}
\operatorname{Li}_2((s, t)^{\alpha} e_{\alpha}) = -\sum_{j \geq 1} \frac{(s, t)^{j \alpha} e_{j\alpha}}{j^2},
\end{equation*}
as an element of $\widehat{\g}$. Using the Lie algebra structure \eqref{classBracket}, one checks that 
\begin{equation}\label{classAction}
\Ad \exp(\Omega \operatorname{Li}_2((s, t)^{\alpha} e_{\alpha}))(e_{\beta}) = e_{\beta} (1 - (s, t)^{\alpha} e_{\alpha})^{\Omega \bra \alpha, \beta\ket}.
\end{equation}
\begin{cor}\label{RH} A solution of \eqref{TBA} solves the Riemann-Hilbert factorization problem with respect to the contour given by the rays $\ell \in \{\ell_{\alpha} : \Omega(\alpha) \neq 0\}$ and Stokes (i.e. jump) factors \begin{equation*}
S_{\ell} = \prod_{Z(\alpha) \in \ell}\Ad \exp(\Omega(\alpha) \operatorname{Li}_2((s, t)^{\alpha} e_{\alpha})).
\end{equation*}
\end{cor}
\begin{proof} This follows immediately from Lemma \ref{JumpLemma} and \eqref{classAction}, by the definition \eqref{TBop} of $\Phi$. Notice that this uses the commutative algebra structure \eqref{classProduct} on $\widehat{\g}$.
\end{proof}
In fact we are only interested in solutions of \eqref{TBA} which are obtained by iteration from $\psi^0$: these are analogues of the iterative solution considered in ibid. Appendix C. Consider the sequence $\psi^{(i)} = \Phi^{(i)}(\psi^0)$ for $i \geq 0$, where $\Phi^{(i)} = \Phi \circ \cdots \circ \Phi$ ($i$ times). To compute $\psi^{(i)}$ we start by rewriting (for all $\psi$)
\begin{align*}  
\sum_{\gamma'} \Omega(\gamma') \log(1 - (s, t)^{\gamma'}\psi_{\gamma'}(\z')) \gamma' &= \sum_{\gamma'}\Omega(\gamma') \sum_p \frac{(s,t)^{p\gamma'}}{p^2} \psi_{p\gamma'}(\z') p\gamma'\\
&= \sum_{\gamma'} (s, t)^{\gamma'} \psi_{\gamma'}(\z') \dt(\gamma') \gamma',
\end{align*}
where
\begin{equation*}
\dt(\gamma') = \sum_{p > 0,\,p | \gamma'} \frac{\Omega(p^{-1}\gamma')}{p^2}.
\end{equation*}
\begin{rmk} The notation DT reflects the usual way in which BPS state counts $\Omega$ are related to Donaldson-Thomas invariants.
\end{rmk}

\noindent So we can rewrite \eqref{TBop} as
\begin{equation*}
\Phi(\psi)_{\alpha}(\z) = \psi^0_{\alpha}(\z)\exp\left(\sum_{\gamma'} \bra \dt(\gamma')\gamma', \alpha \ket \int_{\ell_{\gamma'}}\frac{d\z'}{\z'}\rho(\z, \z')(s, t)^{\gamma'}\psi_{\gamma'}(\z')\right),
\end{equation*}
and we see
\begin{align}\label{TBinduction}
\nonumber \psi^{(i)}_{\alpha}(\z) = \psi^0_{\alpha}(\z)&\sum_{\mathbf{p}}\frac{1}{\mathbf{p}!}\\
&\prod \left(\bra \dt(\gamma'_j)\gamma'_j, \alpha \ket \int_{\ell_{\gamma'_j}}\frac{d\z'}{\z'}\rho(\z, \z')(s, t)^{\gamma'_j}\psi^{(i-1)}_{\gamma'_j}(\z')\right)^{\mathbf{p}_j}, 
\end{align}
where we sum over ordered partitions $\mathbf{p}$, and we take the product over all unordered collections $\{\gamma'_1, \ldots, \gamma'_l\}\subset\Gamma$ for $l$ the length of $\mathbf{p}$. Let us denote by $\T$ a connected rooted tree, decorated by elements $\gamma' \in \Gamma$ (i.e. there is a map from the vertex set $\T^{[0]}$ to $\Gamma$). We denote by $\alpha(v)$ the decoration at $v\in\T^{[0]}$, and introduce a factor 
\begin{equation*}
\W_{\T} = (-1)^{|\T^{[1]}|}\frac{\dt(\gamma_{\T})}{|\Aut(\T)|} \prod_{v \to w} \bra \alpha(v), \dt(\alpha(w))\alpha(w)\ket,
\end{equation*}
where $\gamma_{\T}$ denotes the label of the root of $\T$, and $\Aut(\T)$ is the automorphism group of $\T$ as a decorated, rooted tree. We write $c(\T)$ for the sum $\sum_{v\in\T^{[0]}}\alpha(v)$. To each $\T$ we also attach a ``propagator" $\G_{\T}(\z)$ which is a holomorphic function on $U$ with values in $\widehat{\g}$, defined recursively by
\begin{equation}\label{propagators}
\G_{\T}(\z) = \int_{\ell_{\gamma_{\T}}} \frac{d\z'}{\z'}\rho(\z, \z') \psi^0_{\gamma_{\T}}(\z')\prod_{\T'}\G_{\T'}(\z'),
\end{equation}
where $\{\T'\}$ denotes the set of (connected, rooted, decorated) trees obtained by removing the root of $\T$ (setting $\G_{\emptyset}(\z) = 1$). By applying \eqref{TBinduction} inductively we obtain
\begin{equation*}
\psi^{(i)}_{\alpha}(\z) = \psi^0_{\alpha}(\z)\sum \prod_j \bra \alpha, \gamma_{\T_j}\W_{\T_j}(s, t)^{c(\T_j)}\G_{\T_j}(\z)\ket,  
\end{equation*} 
where we sum over all collections of trees $\{\T_1, \ldots, \T_l\}$ as above, with depth at most $i$. By the definition of the base ring $R_k$, we see that the sequence $\psi^{(i)}$ stabilizes for $i \gg 1$ to a solution $\psi^{\infty}$ of \eqref{TBA}, given explicitly by 
\begin{equation}\label{iterativeSol}
\psi^{\infty}_{\alpha}(\z) = \psi^0_{\alpha}(\z)\exp \bra \alpha, -\sum_{\T} (s,t)^{c(\T)}\gamma_{\T}\W_{\T}\G_{\T}(\z)\ket, 
\end{equation} 
where we sum over all tree of arbitrary depth: however only a finite number give a nonvanishing contribution. This is the analogue of \cite{gmn} equation (C.26) (see also \cite{neitzke} equation (4.12)). 

Abusing slightly the notation of \cite{gmn}, we say that the function $\W_{\T}\G_{\T}(\z)$ is the \emph{instanton correction} attached to a tree $\T$. 

\begin{rmk} In the context of \cite{gmn} these functions encode the corrections to the metric on the moduli space obtained from the $\Xsf_{\alpha}(m, R; \z)$; the corrections are specified by the BPS spectrum and \eqref{TBA}. In the ``1-particle case" of a single vertex $\{\bullet\}$, the coefficients of the Laurent expansion of $\G_{\bullet}(\z)$ can be computed exactly in terms of Bessel functions. One then recovers the Ooguri-Vafa formula for a local hyperk\"ahler metric (see ibid. section 4.3).
\end{rmk}

\noindent Our main concern is the dependence of the operator $\Phi$ (and in particular of our solution $\psi^{\infty}$ to \eqref{TBA}) on the parameter $Z$. Fix central charges $Z^-, Z^+$. Suppose that all the $\gamma_i$ ($\eta_j$) have a single common charge $Z^{\pm}(\gamma_i)$, independent of $i$ (respectively $Z^{\pm}(\eta_j)$, independent of $j$). Let $Z^+(\gamma_i)$ precede $Z^+(\eta_j)$ in the clockwise order, in the first quadrant, while the opposite happens for $Z^-(\gamma_i), Z^-(\eta_j)$. Moreover we assume (without loss of generality) that the cone spanned by $Z^{+}(\gamma_i), Z^{+}(\eta_j)$ contains strictly the cone spanned by $Z^{-}(\gamma_i), Z^{-}(\eta_j)$. We interpolate between $Z^+$ and $Z^-$ with the family $Z^{\tau} = (1 - \tau) Z^+ + \tau Z^-$ (for $\tau \in [0, 1]$). There is a unique $\tau_0 \in [0, 1]$ for which $Z^{\tau_0}(\Gamma)$ is contained in a real line in $\C$.  

Let us prescribe a BPS spectrum for $Z^\tau$, $\tau < \tau_0$ by $\Omega_{Z^\tau}(\pm\gamma_i) = \Omega_{Z^\tau}(\pm\eta_j)=1$ for all $i, j$, while $\Omega$ vanishes on the rest of $\Gamma$. From this we construct our solution $\psi^{\infty}(\z, Z^\tau)$ to \eqref{TBA} for a fixed $\tau < \tau_0$; it is defined in an open set $U^\tau$. Similarly, given a \emph{constant} BPS spectrum $\Omega_{Z^\tau}$ for $\tau > \tau_0$, we obtain solutions $\psi^{\infty}(\z, Z^{\tau})$ on open sets $U^\tau$ ($\tau > \tau_0$).\\ 

\noindent\textbf{Basic problem.} Choose a point $\z$ which lies outside the cone $\bigcup_\tau U^\tau$. Solve for a spectrum $\Omega_{Z^-}$ such that the family of automorphisms $\psi(\z, Z^\tau)$ (which is defined for $\tau \in [0, 1]\setminus\{\tau_0\}$) extends to an element of $C^0([0,1], \Aut_{R_k}(\widehat{\g}))$ for all $R \gg 1$. Let us recall the main results about this problem.
\begin{enumerate}
\item There are always solutions $\Omega_{Z^-}$. Moreover if $(s, t)^{\alpha} \neq 0$ in $R_k$ then $\Omega_{Z^-}(\alpha)$ is unique (and independent of $k$). This follows from the wall-crossing theory for BPS states and standard results about finite-dimensional Riemann-Hilbert factorization problems. In the following we will always write $\Omega_{Z^-}(\alpha)$ for this uniquely determined value (i.e. for $k \gg 1$). 
\item The solution $\Omega_{Z^-}$ can be expressed as a sum over trees $\T$ as above; the contribution\footnote{When $\gamma'$ is not primitive the contribution of $\T$ to $\Omega_{Z^-}(\gamma')$ is not (in general) a number, but rather a more complicated object (a ``singular integral"); such contributions combine to give a genuine $\Q$-valued contribution.} of a tree $\T$ depends only on the instanton correction $\W_{\T}\G_{\T}(\z)$ for $\tau < \tau_0$ (in particular, it vanishes if $\W_{\T}\G_{\T}(\z) \equiv 0$). This is explained in detail in \cite{jacopo}, and it will be implicitely reviewed in what follows. 
\end{enumerate}   
We will refer to the contribution of $\T$ to $\Omega_{Z^-}$ as its \emph{instanton contribution}. As discussed in \cite{jacopo} the set of all instanton contributions is a rather intricate combinatorial object.  The problem we consider here is to obtain a better understanding of the structure of this set. We will show that there is indeed a finer structure, which has a tropical nature. In particular instanton contributions encode (in a very natural way) a set of tropical invariants $N^{\rm trop}({\bf w})$.

\section{Instanton contributions and tropical curves}\label{tropicalSection} We follow the notion of (rational) \emph{plane tropical curve} given in \cite{gama} Definition 2.3. Thus a curve for us means a triple $(\Upsilon, \omega, h)$ where $(\Upsilon, \omega)$ is a connected, simply connected, weighted graph (with a subset of unbounded edges), and $h\!:\Upsilon \to \R^2$ is a proper map from (the topological realization of) $\Upsilon$ to the real plane, such that for each edge $E$, $h(E)$ lies in a line of rational slope. Most importantly, fixing a trivalent vertex $V \in \Upsilon^{[0]}$, with incident edges $E_i$ and $m_V(E_i)$ integral primitive vectors outgoing from $h(V)$ in the direction of $h(E_i)$, we require the \emph{balancing condition}
\begin{equation*}
\sum_{i} \omega(E_i) m_V(E_i)= 0.
\end{equation*}
Similarly we adopt the notion of \emph{degree} for a tropical curve given in ibid. Definition 2.6, as an element of the free abelian group generated by $\Z^2\setminus\{(0,0)\}$ (with addition $\oplus$). We also parallel the definition of the \emph{combinatorial type of a tropical curve} (for brevity, a \emph{tropical type}) contained in ibid. Definition 3.3, with the proviso that we work with unmarked curves. Thus the tropical type $[\Upsilon]$ is the data of the graph $\Upsilon$ and, for each vertex $V$, the triple $\{\omega(E_i) m_V(E_i)\}$ (underlying an actual tropical curve $(\Upsilon, \omega, h)$).\\ 

\noindent  For the type of curves we are interested in we can think of a tropical degree alternatively as a \emph{weight vector} ${\bf w} = ({\bf w}_1, {\bf w}_2)$, where each ${\bf w}_i$ is a collection of integers $w_{ij}$ (for $1 \leq i \leq 2$ and $1 \leq j \leq l_i$) such that  $1 \leq w_{i1} \leq w_{i2} \leq \dots \leq w_{i{l_i}}$. For  $1 \leq j \leq l_1$ choose a general collection of parallel lines $\mathfrak{d}_{1j}$ in the direction $(0, 1)$, respectively $\mathfrak{d}_{2j}$ in the direction $(1, 0)$ for $1 \leq j \leq l_2$. We attach the weight $w_{ij}$ to the line $\mathfrak{d}_{ij}$, and think of the lines $\mathfrak{d}_{ij}$ as ``incoming" unbounded edges for connected, rational tropical curves $\Upsilon \subset \R^2$. We prescribe that such curves $\Upsilon$ have a single additional ``outgoing" unbounded edge in the direction $(|{\bf w}_1|,|{\bf w}_2|)$. Let us denote by $\mathcal{S}(\mathfrak{d}, {\bf w})$ the finite set of such tropical curves $\Upsilon$ (for a general, fixed choice of ends $\mathfrak{d}_{ij}$). We denote by $N^{\rm trop}({\bf w}) = \#^{\mu} \mathcal{S}(\mathfrak{d}, {\bf w})$ the tropical count of curves $\Upsilon$ as above, i.e. the number of elements of $\mathcal{S}(\mathfrak{d}, {\bf w})$ counted with the usual multiplicity $\mu$ of tropical geometry (see \cite{mikhalkin}). It is known that $\#^{\mu} \mathcal{S}(\mathfrak{d}, {\bf w})$ does not depend on the general choice of unbounded edges $\mathfrak{d}_{ij}$ (see \cite{mikhalkin}, \cite{gama}).\\ 

\noindent Fix a (decorated, rooted) tree $\T$. The \emph{degree of} $\T$ is $\deg(\T) = \oplus_{v \in \T^{[0]}} \pi(\alpha(v))$, an element of the free abelian group generated by $\Z^2\setminus\{(0,0)\}$. Notice that we can also think of $\deg(\T)$ alternatively as a weight vector ${\bf w}$. The \emph{charge} of $\T$ is $c(\T) = \sum_{v \in \T^{[0]}} \alpha(v)\in \Gamma$ (as before), and the \emph{reduced charge} of $\T$ as $\bar{c}(\T) = \sum_{v \in \T^{[0]}} \pi(\alpha(v)) \in \Z^2$. Recall that a \emph{labelling} of a tree $\T$ is a bijection $\nu\!:  \{1, \ldots, |\T^{[0]}|\} \to \T^{[0]}$. We allow labellings which are not compatible with the order on $\T$ induced by the choice of a root, namely we do \emph{not} require that if $\nu(i) \to \nu(j)$ in $\T$ then $i < j$.\\

\noindent We regard the instanton correction associated to $\T$ as a function of $\tau \in [0,1]\setminus\{\tau_0\}$, namely $\W_{\T}(\tau)\G_{\T}(\z; \tau)$. The $\tau$ dependence of $\W_{\T}$ comes from thinking of the BPS spectrum as a (locally constant) function of $\tau$, $\Omega_{Z^{\tau}}$. Notice first that $\Omega_{Z^-}(\gamma_i) = 1$. This can be seen by looking at the contribution of a singleton $\{\bullet\}$ labelled by $\gamma_i$ to $\sum_{\T} \W_{\T}\G_{\T}(\z)$, given by $\Omega(\gamma_i, \tau)\gamma_i\int_{\ell_{\gamma_i}}\frac{d\z'}{\z'}\rho(\z, \z')\psi^0_{\gamma_i}(\z', \tau)$. By Lemma \ref{BesselLemma}, choosing $R$ large enough we see that a necessary condition for the continuity of $\psi(\z, Z^{\tau})$ across $\tau_0$ is that the integral is continuous through $\tau_0$, so $\Omega(\gamma_i)$ cannot jump. Similarly $\Omega_{Z^-}(\eta_j) = 1$. It follows that $\W_{\T}(\tau)$ is actually constant. For example in the simple case when all the vertices are labelled by $\gamma_i$ or $\eta_j$, and if we assume without loss of generality that $\gamma_{\T}$ is a (positive) multiple of $\gamma_i$, then $\W_{\T}(\tau) \equiv \pm \frac{\gamma_i}{|\Aut(\T)|}$.

On the other  hand $\G_{\T}(\z)$ depends on $\tau$ both through the integration rays $\ell_{\alpha}(Z^{\tau})$ and the integrands involving $\psi^0_{\alpha}(\z; \tau)$. Suppose that $\W_{\T}(\tau)\G_{\T}(\z; \tau)$ does not vanish for $\tau < \tau_0$. By the definition of $\Omega_{Z^+}$ the vertices of $\T$ are decorated by elements of $\Z \gamma_i$ or $\Z \eta_j$. In the following we assume that they are actually decorated by $\N \gamma_i$ or $\N \eta_j$, and refer to $\T$ as a \emph{positive} tree; this is not restrictive for our purposes. We will also assume that $\gamma_{\T}$ is a (positive) multiple of $\gamma_i$.\\

\noindent Let us fix a sufficiently small $\eps > 0$ and consider the jump 
\begin{equation*}
\G_{\T}(\z, \tau_0 - \eps) - \G_{\T}(\z, \tau_0 + \eps) = J_{\T}(\eps, R) e_{c(\T)}.
\end{equation*}
To compute this we would like first to move all the integration rays $\ell^{+}_{\alpha} = \ell_{\alpha}(\tau_0 - \eps)$ in $\G_{\T}(\z, \tau_0 - \eps)$ to the corresponding rays $\ell^-_{\alpha} = \ell_{\alpha}(\tau_0 + \eps)$. This is of course problematic since in moving the ray $\ell^{+}_{\alpha(v)}$ corresponding to some $v \in \T^{[0]}$ we may cross one or more of the integration rays $\ell^{+}_{\alpha(w)}$ where $v \to w$ or $w \to v$ are edges of $\T$.
\begin{lem}\label{LemmaA} Fix a tree $\T$ and a labelling $\nu$ as above. Suppose that $\bar{c}(\T)$ is primitive.
\begin{enumerate}
\item There is a natural construction (depending on $\nu$) that associates to $\G_{\T}(\z)$ and the critical value $\tau_0$ a set of combinatorial types $\{[\Upsilon_j]\}_{\T}$ of plane rational tropical curves of degree $\deg(\T)$, together with signs $(-1)^{[\Upsilon_j]} \in \{\pm 1\}$. For suitable $\nu$ these combinatorial types (as $\T$ varies) all belong to a single set $[\mathcal{S}(\mathfrak{d}, \deg(\T))]$.
\item $J_{\T}(\tau^{\pm}, R)$ has an asymptotic expansion, as $R \to +\infty$,
\begin{align*}
J_{\T}(\eps, R) = &\sum_{\{[\Upsilon_j]\}_{\T}} (-1)^{[\Upsilon_j]} f(\z, R)\\ &+ o_{\eps}(\frac{1}{\pi |Z^{\tau_0}(c(\T))| R} e^{-\pi |Z^{\tau_0}(c(\T))| R}) + O(\eps),  
\end{align*}
where the $o_{\eps}$ error term is for fixed, sufficiently small $\eps > 0$, the $O(\eps)$ term is uniform in $R \geq 1$ as $\eps \to 0$, and we have, nontangentially to the boundary of the cone spanned by $Z^+(\pm \gamma_i), Z^+(\pm \eta_j)$, 
\begin{equation*}
\lim_{\z \to 0} f(\z, R) \sim \frac{1}{\pi |Z^{\tau_0}(c(\T))| R}e^{-\pi |Z^{\tau_0}(c(\T))| R}
\end{equation*}
for $R \to +\infty$ (holding uniformly for $\eps > 0$ sufficiently small).
\end{enumerate}
\end{lem}
\begin{proof} We start by describing the construction in (1). We will encode the precise way in which $\G_{\T}$ jumps as $\tau$ crosses the critical value $\tau_0$ using again a rooted tree $\mathfrak{T}$. A vertex $u \in \mathfrak{T}^{[0]}$ will correspond to a rooted, decorated, labelled tree $\T_{u}$ obtained as a ``contraction" of the root tree $\mathfrak{T}_0 = \T$ (which correspond to successive applications of the Fubini and residue theorems to $\G_{\T}$). Thus each $\T_u$ parametrizes in turn an iterated integral $\G_u$.\\

\noindent\textbf{Construction of $\mathfrak{T}$.} This is a variant of the iterative process described informally in \cite{jacopo} section 3.2. In particular $\mathfrak{T}$ is obtained inductively from a sequence of trees $\mathfrak{T}_i$, stabilizing to $\mathfrak{T} = \mathfrak{T}_{\infty}$. In what follows the functions $\psi^0$ are always evaluated at $\tau_0 - \eps$. 

The first step is moving the ray $\ell^+_{\alpha(\nu(1))}$ corresponding to $\nu(1) \in \T^{[0]}$ to $\ell^-_{\alpha(\nu(1))}$. By our conventions for $\ell^{+}_{\gamma_i}, \ell^+_{\eta_j}$ and our choice of $\z$ this can be done without picking up residue terms even when $\nu(1)$ is the root of $\T$. Thus $\mathfrak{T}_1$ has the form $u \to u'$, with $\T_{u'} = \T$, but with the factor in $\G_{\T_{u'}}$ corresponding to $v = \nu(1)$ replaced with (writing $v'$ for the unique $v' \to v$ in $\T$)
\begin{equation*}
\int_{\ell^-_{\alpha(v)}} \frac{d\z_v}{\z_v}\rho(\z_{v'}, \z_v)\psi^0_{\alpha(v)}(\z_v).
\end{equation*} 
In particular the leaf of $\mathfrak{T}_1$ is just $\T$ and so it is labelled by our choice $\nu$. We proceed by induction on $i$ and choose a leaf $u \in \mathfrak{T}^{[0]}_i$. We will construct a tree $\mathfrak{T}'_i$, depending on $u$, obtained by extending $\mathfrak{T}_i$ at $u$. We will then define $\mathfrak{T}_{i+1}$ as the union of all $\mathfrak{T}'_i$ as $u$ varies in the leaves of $\mathfrak{T}_i$. By induction $u$ corresponds to a labelled, decorated tree $\T_u$, parametrizing an iterated integral $\G_u$, such that there is a natural bijective correspondence between the factors 
\begin{equation}\label{pieceIntegral}
\int_{\ell} \frac{d\z_v}{\z_v}\rho(\z_{v'}, \z_v)\psi^0_{\alpha(v)}(\z_v)
\end{equation}
(for some ray $\ell \subset \C^*$) appearing in $\G_u$ and the set of vertices $v \in \T^{[0]}_u$ (with $v' \to v$). Indeed this property certainly holds for $\mathfrak{T}_1$, and it is preserved by the inductive step we are going to perform.\\ 
\begin{rmk} In \eqref{pieceIntegral} we may have $\z_{v'} \in \ell$, in the sense that we allow factors of the form
\begin{equation*}
\lim_{\ell' \to \ell} \int_{\ell'} \frac{d\z_v}{\z_v}\rho(\z_{v'}, \z_v)\psi^0_{\alpha(v)}(\z_v)
\end{equation*} 
where $\z_{v'} \in \ell$. However in this case \eqref{pieceIntegral} will be decorated with the direction in which $\ell'$ approaches $\ell$, using $\ell' \to^{\pm} \ell$ for the clockwise (respectively counterclockwise).
\end{rmk}

\noindent We say that a ray $\ell$ \emph{separates} rays $\ell', \ell''$ if they are all contained in the sector generated by $\ell^+_{\gamma_i}, \ell^+_{\eta_j}$, and $\ell', \ell''$ lie in two different connected components of the complement of $\ell$ in this sector.\\ 
\begin{rmk}
We allow the limiting case in which $\ell' \to \ell$ in a different component from $\ell''$, or possibly $\ell'\to\ell$ and $\ell''\to\ell$ in different components. 
\end{rmk}

\noindent Consider the set of vertices $v \in \T^{[0]}_u$ for which one of the following occurs:
\begin{enumerate}
\item the corresponding factor in $\G_u$ is of the form 
\begin{equation*}
\int_{\ell^+_{\alpha(v)}} \frac{d\z_v}{\z_v}\rho(\z_{v'}, \z_v)\psi^0_{\alpha(v)}(\z_v) 
\end{equation*}
where $\alpha(v)$ is a positive multiple of $\gamma_i$ or $\eta_j$, or
\item it is of the form
\begin{equation*}
\int_{\ell} \frac{d\z_v}{\z_v}\rho(\z_{v'}, \z_v)\psi^0_{\alpha(v)}(\z_v)
\end{equation*}
for some $\ell \subset \C^*$ which is not one of $\ell^{\pm}_{\alpha(v)}$. 
\end{enumerate}
In fact we will see (inductively) that there is at most one $v$ for which (2) holds.

If the set of $v$ satisfying (1) or (2) is empty we are done and we set $\mathfrak{T}'_i = \mathfrak{T}_i$.  If not we choose the first such $v$ with respect to the labelling of $\T_u$. Since $\T_u$ is rooted, there is at most one arrow $v' \to v$, and possibly several arrows $v \to v''_j$.
Suppose that $v$ satisfies (1) or (2). Then the factor of $\G_u$ corresponding to $v$ fits into 
\begin{align*}
\int_{\ell^{\pm}_{\alpha(v')}} \frac{d\z_{v'}}{\z_{v'}}\rho(\z_w, \z_{v'}) \psi^0_{\alpha(v')}(\z_{v'})&\int_{\ell} \frac{d\z_v}{\z_v}\rho(\z_{v'}, \z_v)\psi^0_{\alpha}(\z_v)\\
&\prod_{j}\int_{\ell^{\pm}_{\alpha(v''_j)}}\frac{d\z_{v''_j}}{\z_{v''_j}}\rho(\z_v, \z_{v''_j})\psi^0_{\alpha(v''_j)}(\z_{v''_j})
\end{align*}
where $\ell$ is either $\ell^+_{\alpha}$ or a ray distinct from $\ell^-_{\alpha}$, and $w \to v'$. If none of the rays $\ell^{\pm}_{\alpha(v')}$ and $\ell^-_{\alpha(v''_j)}$ separate $\ell$ and $\ell^-_{\alpha}$ we extend $\mathfrak{T}_i$ to $\mathfrak{T}'_{i}$ by a single child $\tilde{u}$ of $u$, with tree $\T_{\tilde{u}}$ isomorphic to $\T_{u}$, and $\G_{\tilde{u}}$ obtained from $\G_{u}$ by replacing $\ell$ in the factor above with $\ell^-_{\alpha(v)}$. Otherwise by Fubini we rewrite the integral above in the form
\begin{align}\label{fubini}
\nonumber  \int_{\ell^{\pm}_{\alpha(v')}} \frac{d\z_{v'}}{\z_{v'}}\rho(\z_w, \z_{v'}) \psi^0_{\alpha(v')}(\z_{v'})&\left(\prod_{j}\int_{\ell^{\pm}_{\alpha(v''_j)}}\frac{d\z_{v''_j}}{\z_{v''_j}}\psi^0_{\alpha(v''_j)}(\z_{v''_j})\right)\\
&\int_{\ell} \frac{d\z_v}{\z_v}\prod_j \rho(\z_v, \z_{v''_j})\rho(\z_{v'}, \z_v)\psi^0_{\alpha(v)}(\z_v)
\end{align}
The function $\frac{1}{\z_v}\prod_j\rho(\z_v, \z_{v''_j})\rho(\z_{v'}, \z_v)\psi^0_{\alpha(v)}(\z_v)$ is holomorphic in $\z_v \in \C^*\setminus\{\z_{v'}, \z_{v''_j}\}$, with simple poles at $\z_{v'}, \z_{v''_j}$ of residues given respectively by $-\frac{1}{2\pi i}\rho(\z_{v'}, \z_{v''_j})\psi^0_{\alpha}(\z_{v'})$ and $\frac{1}{2\pi i}\rho(\z_{v'}, \z_{v''_j})\psi^0_{\alpha(v)}(\z_{v''_j})$. If we apply the residue theorem (justified by Lemma \ref{BesselLemma}) we can rewrite \eqref{fubini} as
\begin{align}\label{push2}
\nonumber  \int_{\ell^{\pm}_{\alpha(v')}} \frac{d\z_{v'}}{\z_{v'}}\rho(\z_w, \z_{v'}) \psi^0_{\alpha(v')}(\z_{v'})&\left(\prod_{j}\int_{\ell^{\pm}_{\alpha(v''_j)}}\frac{d\z_{v''_j}}{\z_{v''_j}}\psi^0_{\alpha(v''_j)}(\z_{v''_j})\right)\\
&\int_{\ell^{-}_{\alpha(v)}} \frac{d\z_v}{\z_v}\prod_j\rho(\z_v, \z_{v''_j})\rho(\z_{v'}, \z_v)\psi^0_{\alpha(v)}(\z_v)
\end{align}
plus residue terms
\begin{equation}\label{res2}
\mp \int_{\ell^-_{\alpha(v')}} \frac{d\z_{v'}}{\z_{v'}}\rho(\z_w, \z_{v'}) (-1)^{\bra \alpha(v), \alpha(v')\ket}\psi^0_{\alpha(v) + \alpha(v')}(\z_{v'})\prod_{j}\int_{\ell^{\pm}_{\alpha(v''_j)}}\frac{d\z_{v''_j}}{\z_{v''_j}}\rho(\z_{v'}, \z_{v''_j})\psi^0_{\alpha(v''_j)}(\z_{v''_j})
\end{equation}
and
\begin{align}\label{res3}
\nonumber \pm \int_{\ell^{\pm}_{\alpha(v')}} \frac{d\z_{v'}}{\z_{v'}}\rho(\z_w, \z_{v'}) \psi^0_{\alpha(v')}(\z_{v'})&\int_{\ell^-_{\alpha(v''_k)}} \frac{d\z_v}{\z_v}\rho(\z_{v'}, \z_v)(-1)^{\bra \alpha(v), \alpha(v''_k)\ket}\psi^0_{\alpha(v) + \alpha(v''_k)}(\z_v)\\
&\prod_{j \neq k}\int_{\ell^{\pm}_{\alpha(v''_j)}}\frac{d\z_{v''_j}}{\z_{v''_j}}\rho(\z_v, \z_{v''_j})\psi^0_{\alpha(v''_j)}(\z_{v''_j})
\end{align}
where \eqref{res2} is only present if $\ell^{\pm}_{\alpha(v')}$ is in fact $\ell^-_{\alpha(v')}$ and morevoer $\ell^-_{\alpha(v')}$ separates $\ell$ and $\ell^-_{\alpha(v)}$, and a term \eqref{res3} appears for each $\ell^-_{\alpha(v''_k)}$ separating $\ell$, $\ell^-_{\alpha(v)}$. The signs in \eqref{res2}, \eqref{res3} are determined according to whether $\ell$ moving to $\ell^{-}_{\alpha(v)}$ crosses $\ell^-_{\alpha(v')}$ (respectively $\ell^-_{\alpha(v''_k)}$) in the clockwise, respectively counterclockwise direction. We extend $\mathfrak{T}_i$ to $\mathfrak{T}'_{i}$ by children  $\tilde{u}, u'$ and $u''_k$ of $u$, where $u'$ is only present if if $\ell^{\pm}_{\alpha(v')} = \ell^-_{\alpha(v')}$ separates $\ell$ and $\ell^-_{\alpha(v)}$, and there is a child $u''_k$ for each $\ell^-_{\alpha(v''_k)}$ separating $\ell$, $\ell^-_{\alpha(v)}$. The tree $\T_{\tilde{u}}$ is just $\T_u$, but with \eqref{push2} attached. To obtain $\T_{u'}$ we remove $v$ from $\T_u$, and connect $v'$ to all $v''_j$. We decorate $v'$ by $\alpha(v) + \alpha(v')$, leave the other decorations unchanged, and use the natural induced labelling. Similarly for $\T_{u''_j}$ we remove $v$ from $\T_u$ and we replace it with $v''_k$ (i.e. we contract the edge $v \to v''_k$), decorated by $\alpha(v) + \alpha(v''_k)$. Again there is a natural induced labelling on $\T_{u''_j}$. By construction our inductive assumptions are all satisfied.\\ 
\begin{rmk} Following our convention, if $\ell^-_{\alpha(v)}$ coincides with $\ell^-_{\alpha(v')}$ or a subset of the $\ell^-_{\alpha(v''_j)}$ (possibly empty), or both, then the integral over $\ell^-_{\alpha(v)}$ in \eqref{push2} is actually a limit of integrals over $\ell \to^{\pm} \ell^-_{\alpha(v)}$.
\end{rmk}\\
\noindent When we go from $\T_u$ to $\T_{\tilde{u}}$ (i.e. in the construction of $\mathfrak{T}'_i$), $\T_{u'}$ or $\T_{u''_j}$ either the number of vertices or the cardinality of the set of vertices satisfying (1) or (2) decreases. If we define $\mathfrak{T}_{i+1}$ as the union of $\mathfrak{T}_i$ with all the $\mathfrak{T}'_i$ along the corresponding leaf, we see that the sequence $\mathfrak{T}_i$ stabilizes to $\mathfrak{T}_{\infty}$ for $i \gg 1$. We let $\mathfrak{T} = \mathfrak{T}_{\infty}$.\\

\noindent\textbf{Estimate for $J_{\T}(\eps, R)$.} According to the construction above there is precisely one leaf $\widehat{u}$ of $\mathfrak{T}$ for which each integration ray $\ell$ appearing in $\G_{\T_{\widehat{u}}}$ coincides with $\ell^{-}_{\gamma_i}$ or $\ell^-_{\eta_j}$. The $\psi^0$ factor appearing in the integral over such a ray is of the form $\psi^0_{\alpha}(\tau_0 - \eps)$ where $\alpha$ is a (positive) multiple of either $\gamma_i$ or $\eta_j$. Since the functions $\psi^0_{\alpha}(\tau)$ are continuous across $\tau_0$,  we have as $\eps \to 0$
\begin{equation*}
\G_{\T_{\widehat{u}}} - \G_{\T}(\tau_0 + \eps)= O(\eps)e_{c(\T)}
\end{equation*}
uniformly for $R \geq 1$.

The other terms in $J_{\T}(\eps, R)$ are in bijection with the leaves $u$ of $\mathfrak{T}$ different from $\widehat{u}$. Let $\pi_u$ denote the unique path from the root of $\mathfrak{T}$ to $u$. Then according to the construction above, $\pi_u$ is of maximal length in $\mathfrak{T}$ if and only if $|\T^{[0]}_u| = 1$, and we have 
\begin{equation*}
\G_{\T_{u}}(\z) = (-1)^{\pi_u} \int_{\ell^-_{c(\T)}} \frac{d\z'}{\z'} \rho(\z, \z') \psi^0_{c(\T)}(\z', \tau_0 - \eps)
\end{equation*}
where $(-1)^{\pi_u}$ is a well defined sign $\pm 1$ attached to $\pi_u$. By Lemma \ref{BesselLemma}, in this case we have, as $R \to +\infty$, uniformly in $\eps > 0$ sufficiently small,  
\begin{equation*}
\G_{\T_{u}}(\z) = f(\z, R)e_{c(\T)} 
\end{equation*}
where along curves that do not become tangent to the boundary of the cone spanned by $Z^{\pm}(\gamma_i), Z^{\pm}(\eta_j)$,
\begin{equation*}
\lim_{\z \to 0} f(\z, R) \sim \frac{1}{\pi |Z^{\tau_0}(c(\T))| R} e^{-\pi R |Z^{\tau_0}(c(\T))|}.
\end{equation*}

If $\pi_u$ is not of maximal length in $\mathfrak{T}$, we claim that for fixed $\eps > 0$, sufficiently small, 
\begin{equation*}
\G_{\T_{u}}(\z) = o_{\eps}(\frac{1}{\pi |Z^{\tau_0}(c(\T))|R} e^{-\pi R |Z^{\tau_0}(c(\T))|})e_{c(\T)}
\end{equation*}
as $R\to +\infty$ (where the $o_{\eps}$ notation is a reminder that the bound is not uniform in $\eps$). Indeed by the construction of $\mathfrak{T}$ in this case we have $|\T^{[0]}_u| > 1$, and
\begin{equation}\label{remainder}
\G_{\T_u}(\z) = \int_{L} \prod_{\{w \to v\}\subset \widetilde{T}_u} \frac{d\z_v}{\z_v}\frac{\z_v + \z_w}{\z_v - \z_w} \psi^0_{\alpha(v)}(\z_v),
\end{equation}
where $L$ is the product of all rays $\ell^-_{\alpha(v)}$, $\widetilde{\T}_u$ denotes the augmented tree $w_0 \to \T_u$ (mapping to the root of $\T_u$), with $\z_{w_0} = \z$. Suppose that in \eqref{remainder} we have $\ell^-_{\alpha(v)} \neq \ell^-_{\alpha(w)}$ if $w \to v$. Iterating Lemma \ref{BesselLemma} sufficiently many times, we have that for $R \gg 1$ (depending on $\eps$) $\G_{\T_u}$ is a multiple of $e_{c(\T)}$ bounded by
\begin{equation*}
C \prod_{v \in \T^{[0]}_{u}} \frac{1}{\pi R |Z^+(\alpha(v))|} e^{-\pi R |Z^+(\alpha(v))|}, 
\end{equation*}
where the constant $C$ may be chosen uniformly for $|\T^{[0]}_u|$ bounded. The claim follows since for fixed $\eps > 0$ we have
\begin{equation*}
\sum_{v \in \T^{[0]}_{u}} |Z^+(\alpha(v))| > |Z^+(c(\T_u))|. 
\end{equation*} 
In general we may well have $\ell^-_{\alpha(v)} = \ell^-_{\alpha(w)}$ for some $w \to v$ in \eqref{remainder}, but there is at least one $w' \to v'$ with $\ell^-_{\alpha(v')} \neq \ell^-_{\alpha(w')}$; this holds since we are assuming that $\bar{c}(\T)$ is primitive. Then $\G_{\T_u}(\z)$ is a multiple of $e_{c(\T)}$ bounded by
\begin{equation*}
C' \frac{1}{\pi R |Z^+(\alpha(w'))|} e^{-\pi R |Z^+(\alpha(w'))|} \frac{1}{\pi R |Z^+(c(\T) - \alpha(w'))|} e^{-\pi R |Z^+(c(\T) - \alpha(w'))|}. 
\end{equation*}
As before the claim follows since for fixed $\eps > 0$ we have
\begin{equation*}
|Z^+(\alpha(w'))| + |Z^+(c(\T) - \alpha(w'))| > |Z^+(c(\T_u))|. 
\end{equation*}

\noindent\textbf{Construction of tropical types $[\Upsilon]$.} Let $\pi_u$ denote a path of maximal length in $\mathfrak{T}$ as before. We can think of this as a sequence of trees $\T_i$ for $i = 0, \cdots, N$, where $\T_0 = \T$ (as decorated, labelled trees), and $\T_{N} = \T_u$ contains a single vertex. Moreover by the construction of $\mathfrak{T}$ there are natural maps 
\begin{equation*}
\varphi_i\!: \T^{[0]}_i \to \T^{[0]}_{i+1},
\end{equation*}
such that $\varphi_i$ is either a bijection, or maps two vertices $v_1, v_2 \in \T^{[0]}_i$ to the same vertex $v \in \T^{[0]}_{i+1}$ (and is a bijection on $\T^{[0]}_i \setminus \{v_1, v_2\}$). The set of vertices $\bigcup_i \T^{[0]}_i$ and maps $\{\varphi_j\}$ can be thought of naturally as a tree $\widetilde{\Upsilon}$, whose internal vertices are either $2$-valent or $3$-valent. We define a new unbounded tree $\Upsilon$ (i.e. a tree with a set of half-open, ``infinite" edges) by replacing all subtrees of $\widetilde{\Upsilon}$ of type $A_{K}$ for $K > 1$ with a copy of $A_1$, and by removing all the $1$-valent vertices. 

By the definition of $\mathfrak{T}$ and $\Upsilon$, an edge $E \in \Upsilon^{[1]}$ corresponds to a set of factors
\begin{equation*}
\{\int_{\ell_i} \frac{d\z_{v_i}}{\z_{v_i}}\rho(\z_{v'_i}, \z_{v_i}) \psi^0_{\alpha}(\z_{v_i})\}_i
\end{equation*}
where $\alpha \in \Gamma$ does not depend on $i$. Fix a vertex $V \in \Upsilon^{[0]}$. Then again by the definition of $\mathfrak{T}$ and $\Upsilon$, we can label the edges incident in $V$ as $E_1, E_2, E_{\rm out}$, and we know that for $k = 1, 2$ there are factors 
\begin{equation*}
\int_{\ell_k} \frac{d\z_{v_k}}{\z_{v_k}}\rho(\z_{v'_k}, \z_{v_k}) \psi^0_{\alpha_k}(\z_{v_k})
\end{equation*}
in the sets corresponding to $E_k$, for which the set corresponding to $E_{\rm out}$ contains the factor
\begin{equation*}
\int_{\ell} \frac{d\z_{v'}}{\z_{v}}\rho(\z_{v'}, \z_{v}) (-1)^{\bra \alpha_1, \alpha_2\ket}\psi^0_{\alpha_1 + \alpha_2}(\z_{v_k})
\end{equation*}
obtained as a residue term from pushing $\ell_2$ to $\ell^-_{\alpha_2}$ across $\ell_1 = \ell^-_{\alpha_1}$. We then let the triple of integral vectors $\{\omega(E_1)m_V(E_1), \omega(E_2)m_V(E_2), \omega(E_{\rm out})m_V(E_{\rm out})\}$ corresponding to $V$ be $\{\pi (-\alpha_1), \pi (-\alpha_2), \pi(\alpha_1 + \alpha_2)\}$. Then the crucial relations
\begin{equation}\label{tropCond1}
\omega(E_1)m_V(E_1) + \omega(E_2)m_V(E_2) + \omega(E_{\rm out})m_V(E_{\rm out} = 0
\end{equation}
and
\begin{equation}\label{tropCond2}
m_V(E_1) \not\parallel m_V(E_2) 
\end{equation}
hold. By \eqref{tropCond1} and \eqref{tropCond2} the pair $(\Upsilon, p_{\Upsilon})$ determines the combinatorial type of a connected tropical curve $[\Upsilon]$ in $\R^2$, which is precisely of the kind occurring in a set $\mathcal{S}(\mathfrak{d}, \deg(\T))$. By construction we can choose the original labelling $\nu$ so that all combinatorial types $[\Upsilon]$ arising from $\T$ are compatible with a given total order of $\deg(\T)$, thought of as a set. One can check by induction on $\deg(\T)$ that for suitable $\nu$, for a generic choice of ends $\mathfrak{d}$, the set $[\mathcal{S}(\mathfrak{d}, \deg(\T))]$ contains all $[\Upsilon]$.\\

\noindent For each leaf $u \in \mathfrak{T}$ of maximal depth, we have constructed a tropical type $[\Upsilon](u)$. We set
\begin{equation*}
(-1)^{[\Upsilon](u)} = (-1)^{\pi_u}.
\end{equation*}
Therefore the leading order term in the expansion for $J_{\T}(\eps, R)$ obtained above can be written as
\begin{equation*}
\sum_{u} (-1)^{[\Upsilon](u)} f(\z, R), 
\end{equation*} 
as claimed in (2) of the statement of the present Lemma.
\end{proof}
\begin{exm} Consider the tree $\T = \{\gamma_1 \to \eta_1 \to \gamma_1 \to 2\eta_1\}$
and fix the unique labelling which is compatible with the orientation. Then the (signed) tropical types obtained from $\G_{\T}$ are pictured in Figure \ref{tropicalPicture1}: they comprise the types of a smooth (left) and a nodal (right) tropical curve.
\begin{figure}[ht]
\centerline{\includegraphics[scale=.6]{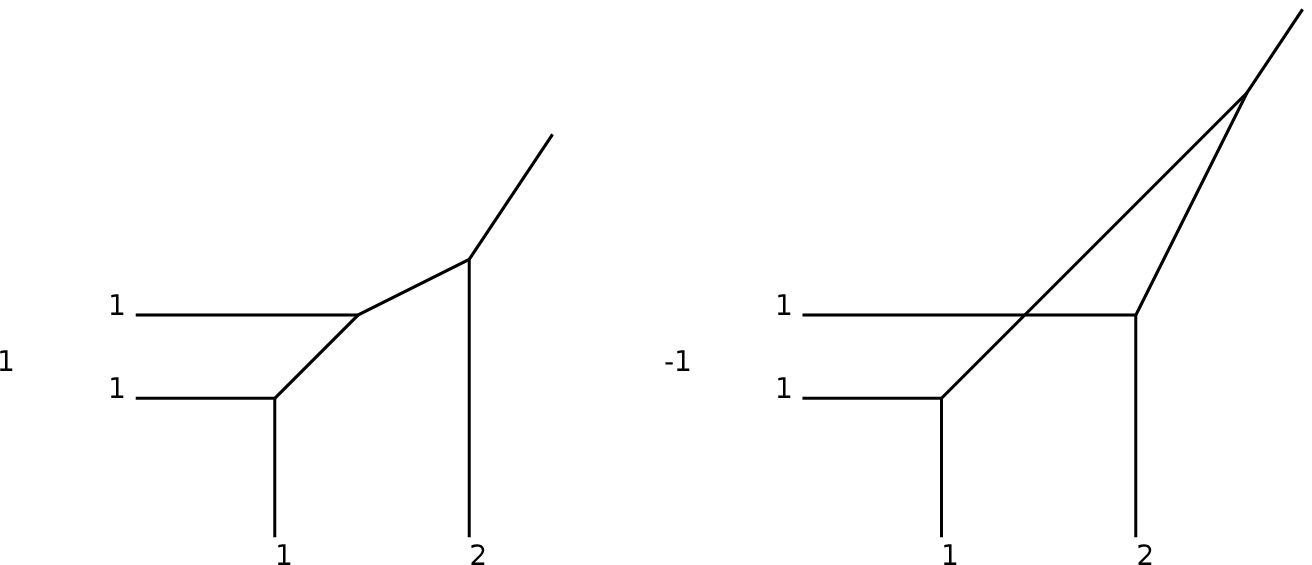}}
\caption{Tropical types for $\G_{\T}$.}
\label{tropicalPicture1}
\end{figure}
\end{exm}
\begin{cor}\label{wallcross} Choose $i \in \{1, \dots, \ell_1\}$. Then we have
\begin{equation*}
\dt_{Z^{-}}(\alpha) = \sum_{c(\T) = \alpha, \gamma_{\T} = p_{\T} \gamma_i} \frac{p_{\T}}{c(\T)_{\gamma_i}}\W_{\T} \sum_{\{[\Upsilon_j]\}_{\T}} (-1)^{[\Upsilon_j]}, 
\end{equation*}
where we let $c(\T)_{\gamma_i}$ denote the component of $c(\T)$ along $\gamma_i$.
\end{cor}
\begin{proof} Choose $k$ such that $(s, t)^{\alpha} \neq 0$ in $R_k$. By the expansion (2) in the statement of Lemma \ref{LemmaA}, and the form of the iterative solution $\psi^{\infty}(\z)$ of \eqref{TBA} given in \eqref{iterativeSol}, applied to the spectrum $\Omega_{Z^-}$, we see that the equality claimed above is a necessary condition for the glueing of the $\psi^{\infty}(\z)$ relative to $\Omega_{Z^+}$ and $\Omega_{Z^-}$ to be continuous. 
\end{proof}
\begin{thm}\label{TheoremB}  The sum over instanton contributions for trees $\T$ labelled by two charges $\gamma, \eta$, of fixed tropical degree and root label, encodes a tropical count, 
\begin{equation*}
\sum_{\deg{\T} = {\bf w}, \gamma_{\T} = w_{11}\gamma} \W_{\T} \left(\sum_{\{[\Upsilon_j]\}_{\T}} (-1)^{[\Upsilon_j]}\right) = \prod_{i, j}\frac{1}{w^2_{i j}}\frac{N^{\rm trop}({\bf w})}{\Aut({\bf w}')}, 
\end{equation*}
(where ${\bf w}'$ is obtained by forgetting $w_{11}$).
\end{thm}
\begin{proof} Let ${\bf w}$ be a weight vector. Write $l_i$ for the length of ${\bf w}_i$. We make the special choice of parameters $\ell_1 = l_1$, $\ell_2 = l_2$, fix a corresponding base ring $R_k$ and consider the problem \eqref{TBA} over $R_k$ for $\tau > \tau_0$ with a deformed BPS spectrum $\Omega_{Z^+}(\pm\gamma_i) = \Omega_{Z^+}(\pm\eta_j) = \epsilon$ (while the other invariants vanish),
\begin{align*}  
\nonumber\psi_{\alpha}(\z) = \psi^0_{\alpha}(\z)\exp&\left(\sum^{l_1}_{i = 1}\bra \epsilon(\pm \gamma_i),\alpha \ket \int_{\ell_{\pm\gamma_i}}\frac{d\z'}{\z'}\rho(\z,\z')\log(1 - s_i\psi_{\pm \gamma_i}(\z'))\right.\\
&\left. + \sum^{l_2}_{j = 1}\bra \epsilon(\pm \eta_j),\alpha \ket \int_{\ell_{\pm\eta_j}}\frac{d\z'}{\z'}\rho(\z,\z')\log(1 - t_j\psi_{\pm \eta_j}(\z')) \right),
\end{align*} 
The invariant 
\begin{equation*}
\dt_{Z^-}(\sum^{l_1}_{i = 1} |w_{1i}| \gamma_i + \sum^{l_2}_{j = 1} |w_{2j}| \eta_j)
\end{equation*}
is now a polynomial in $\epsilon$. Following the argument leading to the expansion \eqref{iterativeSol} and to Corollary \ref{wallcross}, the coefficient $\dt_{Z^-}(\sum^{l_1}_{i = 1} |w_{1i}| \gamma_i + \sum^{l_2}_{j = 1} |w_{2j}| \eta_j)[\epsilon^{l_1 + l_2}]$ can only receive contributions from diagrams $\T$ at $Z^+$ with $l_1 + l_2$ vertices, and such that the corresponding monomial in the variables $s_i, t_j$ is 
\begin{equation*}
\prod^{l_1}_{i = 1} s^{w_{1 i}}_i \prod^{l_2}_{j = 1} t^{w_{2 j}}_j.
\end{equation*}
This is the set $\mathcal{P}$ of diagrams for which $\T^{[0]}$ has cardinality $l_1 + l_2$ and is decorated by 
\begin{equation*}
\{w_{11}\gamma_1, \dots, w_{1l_1}\gamma_{l_1}, w_{21}\eta_1, \dots, w_{2l_2}\eta_{l_2}\}.
\end{equation*} 
Then
\begin{equation*}
\dt_{Z^-}(\sum^{l_1}_{i = 1} |w_{1i}| \gamma_i + \sum^{l_2}_{j = 1} |w_{2j}| \eta_j)[\epsilon^{l_1 + l_2}] = \sum_{\T\in \mathcal{P}, \gamma_{\T} = w_{11}\gamma_1} \W_{\T}\left(\sum_{\{[\Upsilon_j]\}_{\T}} (-1)^{[\Upsilon_j]}\right).
\end{equation*}
We make the change of variables
\begin{align}\label{deformation}
s_i = \sum^k_{r = 1} u_{i r},\quad t_j = \sum^k_{r = 1} v_{j r}, 
\end{align}
where the $u_{i j}, v_{i j}$ satisfy $u^2_{i j} = v^2_{i j} = 0$. The corresponding equation is
\begin{align}\label{TBAdeformed} 
\nonumber\psi_{\alpha}(\z) = \psi^0_{\alpha}(\z)\exp&\left(\sum^{l_1}_{i = 1}\bra \epsilon(\pm\gamma_i),\alpha \ket \int_{\ell_{\pm\gamma_i}}\frac{d\z'}{\z'}\rho(\z,\z')\log(1 - (\sum^k_{r = 1} u_{i r})\psi_{\pm\gamma_i}(\z'))\right.\\
&\left. + \sum^{l_2}_{j = 1}\bra \epsilon(\pm\eta_j),\alpha \ket \int_{\ell_{\pm\eta_j}}\frac{d\z'}{\z'}\rho(\z,\z')\log(1 - (\sum^k_{r = 1} v_{j r})\psi_{\pm\eta_j}(\z')) \right),
\end{align}
where now $\psi_{\alpha}(\z)$ takes values in $\g \otimes \widetilde{R}_k$, where 
\begin{equation*}
\widetilde{R}_k = \frac{\C[u_{i r}, v_{j r}]}{(u^2_{i r}, v^2_{j r})},
\end{equation*}
for $i=1,\dots,l_1, \,\, j=1,\dots,l_2, \,\, r =1,\dots ,k$. Notice that \eqref{deformation} induces an embedding $R_{k} \hookrightarrow \widetilde{R}_k$, so that a solution to \eqref{TBAdeformed} also gives a solution for the equation over $R_k$. Using $u^2_{ij } = 0$ we see that
\begin{align*}
\epsilon\log(1 - (\sum^k_{r = 1} v_{i r})\psi_{\pm\gamma_i}) & = \sum_{p \geq 1} \sum_{|J| = p} p! \prod_{r\in J} \epsilon u_{i r} \frac{\psi_{\pm p\gamma_i}}{p} = \sum_{p \geq 1} \sum_{|J| = p} \log(1 - (p - 1)! \prod_{r\in J} \epsilon u_{i r} \psi_{\pm p\gamma_i}),\\
\epsilon\log(1 - (\sum^k_{r = 1} u_{j r})\psi_{\pm\eta_j}) &= \sum_{p \geq 1} \sum_{|J| = p} p! \prod_{r\in J} \epsilon v_{j r} \frac{\psi_{\pm p\eta_j}}{p} = \sum_{p \geq 1} \sum_{|J| = p} \log(1 - (p - 1)! \prod_{r\in J} \epsilon v_{j r} \psi_{\pm p\eta_j}).
\end{align*}
Then \eqref{TBAdeformed} becomes
\begin{align}\label{TBAdeformed2} 
\nonumber\psi_{\alpha}(\z) = \psi^0_{\alpha}(\z)\exp&\left(\sum^{l_1}_{i = 1} \sum_{p \geq 1} \sum_{|J| = p} \bra (\pm \gamma_i),\alpha \ket \int_{\ell_{\pm \gamma_i}}\frac{d\z'}{\z'}\rho(\z,\z') \log(1 - (p - 1)! \prod_{r\in J} \epsilon u_{i r} \psi_{\pm p\gamma_i}(\z'))\right.\\
&\left. + \sum^{l_2}_{j = 1}\sum_{p \geq 1} \sum_{|J| = p}\bra (\pm\eta_j),\alpha \ket \int_{\ell_{\pm\eta_j}}\frac{d\z'}{\z'}\rho(\z,\z') \log(1 - (p - 1)! \prod_{r\in J} \epsilon v_{j r} \psi_{\pm p\eta_j}(\z'))\right).
\end{align}

We will compute $\dt_{Z^-}(\sum^{l_1}_{i = 1} |w_{1i}| \gamma_i + \sum^{l_2}_{j = 1} |w_{2j}| \eta_j)[\epsilon^{l_1 + l_2}]$ in a different way, using the theory developed in \cite{gps} sections 1 and 2. To see this we consider the two different Riemann-Hilbert problems solved by the iterative solutions $\psi^{\infty}(\z, \tau_0 \pm \eps)$ of \eqref{TBA}, again with the choice $\Omega_{Z^+}(\pm\gamma_i) = \Omega_{Z^+}(\pm\eta_j) = \epsilon$. For $\psi^{\infty}(\z, \tau_0 - \eps)$ we have rays and Stokes factors
\begin{equation}\label{scatter1}
\{(\ell_{\pm \gamma_i}(\tau_0 + \eps), \Ad \exp(\epsilon\operatorname{Li}_2(s_i e_{\pm \gamma_i})), (\ell_{\pm \eta_j}(\tau_0 + \eps), \Ad \exp(\epsilon\operatorname{Li}_2(t_j e_{\pm \eta_j}))\},
\end{equation} 
while for $\psi^{\infty}(\z, \tau_0 + \eps)$ they are given by
\begin{equation}\label{scatter2}
\{(\ell_{\alpha}(\tau_0 - \eps), \prod_{Z(\alpha) \in \ell}\Ad \exp(\Omega_{Z^-}(\alpha; \epsilon) \operatorname{Li}_2(s, t)^{\alpha} e_{\alpha})))\}.
\end{equation}
where $\alpha = \sum^{\ell_1}_{i = 1} a_i \gamma_i + \sum^{\ell_2}_{j = 1} b_j \eta_j$ with $|a_i|, |b_j| \leq k$. The limit $\psi^{\infty}(\z, \tau_0 + 0)$ extends to a holomorphic function on $\C\setminus \R Z^0(\Gamma)$, the complement in $\C^*$ of two rays $\pm \ell$ given by $\R_{\mp} Z^0(\Gamma)$. Furthermore $\psi^{\infty}(\z, \tau_0 + 0)$ solves a Riemann-Hilbert factorisation problem for $\pm \ell$, the Stokes factor of $\ell$ being
\begin{equation*}
S^+_{\ell} = \prod_{j} \Ad \exp(\epsilon \operatorname{Li}_2(t_j e_{\eta_j})) \prod_{i} \Ad \exp(\epsilon \operatorname{Li}_2(s_i e_{\gamma_i})).
\end{equation*}
Similarly the limit $\psi^{\infty}(\z, \tau_0 - 0)$ solves a factorisation problem with factor along $\ell$ given by
\begin{equation*}
S^-_{\ell} = \prod^{\to} \Ad \exp(\Omega_{Z^-}(\alpha; \epsilon) \operatorname{Li}_2 (s, t)^{\alpha} e_{\alpha})
\end{equation*}
where $\alpha = \sum^{\ell_1}_{i = 1} a_i \gamma_i + \sum^{\ell_2}_{j = 1} b_j \eta_j$ with $0 \leq a_i, b_j \leq k$ and we are writing the operators from left to right in the clockwise order of $Z^+(\alpha)$ (this is straightforward since only a finite number of rays appear). Since we know that $\psi^{\infty}(\z, \tau_0 + 0) = \psi^{\infty}(\z, \tau_0 - 0)$ in $\C\setminus \R Z^0(\Gamma)$ it follows that 
\begin{equation}\label{boundaryCond}
S^+_{\ell} = S^-_{\ell}.
\end{equation}
In fact all the invariants $\Omega_{Z^-}(\alpha; \epsilon)$ are uniquely determined by this equality. By the above discussion, we know that $\psi^{\infty}(\z, \tau_0 + \eps)$ is induced by a solution of \eqref{TBAdeformed2}. The corresponding limit Riemann-Hilbert problem for $\psi^{\infty}(\z, \tau_0 + \eps)$ has a Stokes factor along $\ell$ given by
\begin{equation*}
S^+_{\ell} = \prod_j \prod_{p \geq 1} \prod_{|J| = p} \exp(\operatorname{Li}_2((p - 1)! \prod_{r\in J} \epsilon v_{j r} e_{p\eta_j})) \prod_i \prod_{p \geq 1} \prod_{|J| = p} \exp(\operatorname{Li}_2((p - 1)! \prod_{r\in J} \epsilon u_{i r} e_{p\gamma_i}))
\end{equation*}
Now we make the specialization $e_{\gamma_i} = y$, $e_{\eta_j} = x$, and pass to the terminology for operators introduced in \cite{gps} section 0.1. Then we can write
\begin{equation*}
S^+_{\ell} = \theta_{(1, 0), f_1} \circ \theta_{(0, 1), f_2} 
\end{equation*}
where
\begin{equation*}
f_1 = \prod_j \prod_{p \geq 1} \prod_{|J| = p} (1 - (p - 1)! \prod_{r\in J} \epsilon v_{j r} x^p),\quad f_2 = \prod_i \prod_{p \geq 1} \prod_{|J| = p} (1 - (p - 1)! \prod_{r\in J} \epsilon u_{i r} y^p). 
\end{equation*}
So the (specialization of) \eqref{boundaryCond} becomes
\begin{equation*}
\theta_{(1, 0), f_1} \circ \theta_{(0, 1), f_2} = \prod^{\to}_{(a, b)} \theta_{(a, b), f_{(a, b)}}, 
\end{equation*}
and $\dt_{Z^-}(\sum^{l_1}_{i = 1} |w_{1i}| \gamma_i + \sum^{l_2}_{j = 1} |w_{2j}| \eta_j)[\epsilon^{l_1 + l_2}]$ appears as the coefficient of 
\begin{equation*}
\epsilon^{l_1 + l_2} \left(\prod_{i, j}\sum_{|J| = w_{1i}}\sum_{|J'| = w_{2 j}} \prod_{r\in J} |J|! u_{i r} \prod_{r'\in J'} |J'|! v_{j r'} \right) x^a y^b
\end{equation*}
in the function $\log f_{(a, b)}$, where $a = |{\bf w}_2|, b = |{\bf w}_1|$ (this is a standard computation, see e.g. \cite{ks} section 2.5, where $\dt(\alpha)$ is denoted by $-a(\alpha)$). We can calculate this coefficient in terms of tropical geometry using the same argument as in the proof of \cite{gps} Theorem 2.4, but keeping track of the parameter $\epsilon$. Following that procedure one introduces an auxiliary weight vector ${\bf W}$, with $l_1 + l_2$ components, given by
\begin{equation*}
{\bf W} = ((w_{11}), \dots, (w_{1 l_1}), (w_{21}), \dots, (w_{2 l_2})).
\end{equation*}
Although we only defined the tropical invariants $N^{\rm trop}({\bf w})$ for two-components weight vectors ${\bf w}$, as explained in \cite{gps} section 2.3, there is an obvious extension to an arbitrary number of components (with corresponding directions for the infinite ends). The directions attached to the first $l_1$ parts of ${\bf W}$ is $(0, 1)$, respectively $(1, 0)$ for the last $l_2$. Then the argument in loc.\! cit.\! gives
\begin{equation*}
\dt_{Z^-}(\sum^{l_1}_{i = 1} |w_{1i}| \gamma_i + \sum^{l_2}_{j = 1} |w_{2j}| \eta_j)[\epsilon^{l_1 + l_2}] = \prod_{i, j}\frac{1}{w^2_{i j}} N^{\rm trop}({\bf W}), 
\end{equation*}
(using $|\Aut({\bf W})| = 1$). On the other hand, since ${\bf W}$ is just a subdivision of ${\bf w}$, we have 
\begin{equation*}
N^{\rm trop}({\bf W}) = N^{\rm trop}({\bf w}),
\end{equation*}
and therefore
\begin{equation}\label{lifted}
\sum_{\T\in \mathcal{P}, \gamma_{\T} = w_{11}\gamma_1}\W_{\T}\left(\sum_{\{[\Upsilon_j]\}_{\T}} (-1)^{[\Upsilon_j]}\right) = \prod_{i, j}\frac{1}{w^2_{i j}} N^{\rm trop}({\bf w}).
\end{equation}
Consider now the problem with $\ell_1 = \ell_2 = 1$. Then there is a class of instanton corrections labelled by the set $\widetilde{\mathcal{P}}$ of trees $\widetilde{\T}$ for which $\widetilde{\T}^{[0]}$ has cardinality $l_1 + l_2$ and is decorated by $\{w_{11}\gamma, \dots, w_{1l_1}\gamma, w_{21}\eta, \dots, w_{2l_2}\eta\}$, and such that moreover $\gamma_{\widetilde{\T}} = w_{11}\gamma$. The obvious forgetful map $\mathcal{P}\cap\{\gamma_{\T} = w_{11} \gamma_1\} \to \widetilde{\mathcal{P}}$ is onto, and $\W_{\T}\left(\sum_{\{[\Upsilon_j]\}_{\T}} (-1)^{[\Upsilon_j]}\right)$ is constant along the fibres. The fibre over $\widetilde{\T}$ has cardinality $\Aut({\bf w}'){\Aut(\widetilde{\T})}^{-1}$, where ${\bf w}'$ is obtained by dropping $w_{11}$. Moreover $\W_{\widetilde{\T}} = \W_{\T}\Aut(\widetilde{\T})$, so we can rewrite \eqref{lifted} as
\begin{equation*}
\sum_{\widetilde{\T} \in \widetilde{\mathcal{P}}}\W_{\widetilde{\T}}\left(\sum_{\{[\Upsilon_j]\}_{\widetilde{\T}}} (-1)^{[\Upsilon_j]}\right) = \prod_{i, j}\frac{1}{w^2_{i j}}\frac{N^{\rm trop}({\bf w})}{\Aut({\bf w}')}.
\end{equation*} 
 
\end{proof}
\begin{exm} We illustrate this result when ${\bf w} = (1 + 1, 1 + 2)$. Consider the GMN diagrams and their contributions,  
\begin{center} 
\centerline{
\xymatrix{\T_1 = \{\gamma \ar[r] & \ar[r] \eta \ar[r] & \gamma \ar[r] & 2\eta\}\, \sim 0, \quad \T_2 = \{\gamma & \ar[l] \eta & \ar[l] \gamma \ar[r] & 2\eta\}\, \sim 0,\\}
}
\end{center}
\begin{center} 
\centerline{
\xymatrix{\T_3 = \{\gamma \ar[r] & \ar[r] 2\eta \ar[r] & \gamma \ar[r] & \eta\}\, \sim 1,\quad \T_4 = \{ \gamma & \ar[l] 2\eta & \ar[l] \gamma \ar[r] & \eta\}\, \sim 1.}
}
\end{center}
We have $|\Aut({\bf w}')| = 1$ and $\prod w^2_{i j} = 4$ so, by Theorem \ref{TheoremB}, $N^{\rm trop}(1+1, 1+2) = 8$. On the other hand we can compute $N^{\rm trop}(1+1, 1+2) = 8$ by choosing (weighted) ends $\mathfrak{d}_{ij}$ and curves as shown in Figure \ref{Ntrop23}. 
\begin{figure}[ht]
\centerline{\includegraphics[scale=.5]{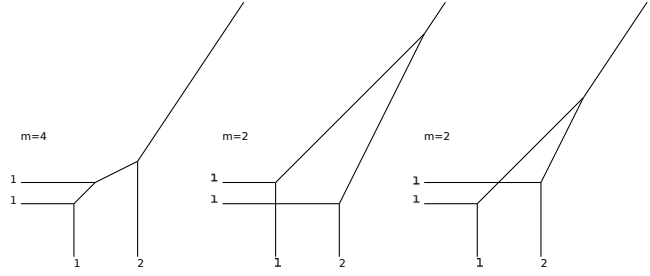}}
\caption{$N^{\rm trop}(1+1, 1+2) = 8$}
\label{Ntrop23}
\end{figure}
Let $[\Upsilon_1], [\Upsilon_2]$ be the two distinct tropical types appearing in the figure: we write $[\Upsilon_1]$ for the smooth type and $[\Upsilon_2]$ for the nodal one. Then one checks that for a unique choice of $\nu_i$, we have $\prod w^2_{i j} \W_{\T_1}\T_1 \sim \prod w^2_{i j} \W_{\T_2}\T_2 \sim -2[\Upsilon_1] + 2[\Upsilon_2]$, $\prod w^2_{i j} \W_{\T_3}\T_3 \sim \prod w^2_{i j} \W_{\T_4}\T_4 \sim 4 [\Upsilon_1]$, so on the GMN side the total cycle of tropical types is $4[\Upsilon_1] + 4[\Upsilon_2]$. This is the same as the set $[\mathcal{S}(\mathfrak{d}, {\bf w})]$ for the choice in Figure \ref{Ntrop23}.  
\end{exm}

\begin{rmk} Theorem \ref{TheoremB} covers the case of diagrams labelled by $\gamma, \eta$ with $\bra \gamma, \eta\ket = 1$. However this is not really a loss of generality with respect to the case of diagrams labelled by $\gamma, \eta$ with arbitrary $\bra \gamma, \eta\ket = \kappa > 0$. In the latter case one still has a weight $\mathcal{W}_{\T, \kappa}$, with the simple relation $\W_{\T, \kappa} = \kappa^{|\T^{[1]}|}  \W_{\T}$. Moreover following the proof of Lemma \ref{LemmaA} and using the definition of the tropical multiplicity $\mu$ we see that for each $[\Upsilon_j] \in \{[\Upsilon_j]\}_{\T}$ we have 
\begin{equation*}
((-1)^{[\Upsilon_j]})_{\kappa} = (-1)^{(\kappa - 1)\mu([\Upsilon_j])}(-1)^{[\Upsilon_j]}.
\end{equation*}
Therefore the analogue of Theorem \ref{TheoremB} is 
\begin{equation*}
\sum_{\deg{\T} = {\bf w}, \gamma_{\T} = w_{11}\gamma} \W_{\T, \kappa} \left(\sum_{\{[\Upsilon_j]\}_{\T}} (-1)^{(\kappa - 1)\mu([\Upsilon_j])}((-1)^{[\Upsilon_j]})_{\kappa}\right) = \kappa^{-|\T^{[1]}|}\prod_{i, j}\frac{1}{w^2_{i j}}\frac{N^{\rm trop}({\bf w})}{\Aut({\bf w}')}. 
\end{equation*}
\end{rmk}
\section{$q$-deformation}\label{qSection}
Kontsevich and Soibelman \cite{ks} deform $\g$ to an \emph{associative, noncommutative} algebra $\g_q$ over $\C[q^{\pm \frac{1}{2}}]$, generated by $\hat{e}_{\gamma}, \gamma \in \Gamma$. The classical product \eqref{classProduct} is quantized to
\begin{equation}\label{qProduct}
\hat{e}_{\alpha} \hat{e}_{\beta} := q^{\frac{1}{2}\bra\alpha, \beta\ket}\hat{e}_{\alpha + \beta}.
\end{equation}
It follows that the classical limit is $q^{\frac{1}{2}} \to -1$. In the quantization the Lie bracket is the natural one given by the commutator. In other words we are now thinking of the $\hat{e}_{\gamma}$ as \emph{operators} (as opposed to the classical bracket \eqref{classBracket}, which corresponds to a Poisson bracket of the $e_{\gamma}$ seen as \emph{functions}). Namely, we set
\begin{equation}\label{qBracket}
[\hat{e}_{\alpha}, \hat{e}_{\beta}] := (q^{\frac{1}{2}\bra\alpha, \beta\ket} - q^{-\frac{1}{2}\bra\alpha, \beta\ket})\hat{e}_{\alpha + \beta}. 
\end{equation}
Since this is the commutator bracket of an associative algebra, $\g_q$ is automatically Poisson.
 
\noindent\textbf{Remark.} After rescaling, the Lie bracket \eqref{qBracket} has the classical limit \eqref{classBracket}:
\begin{equation*}
\lim_{q^{\frac{1}{2}}\to -1} \frac{1}{q-1} [\hat{e}_{\alpha}, \hat{e}_{\beta}] = (-1)^{\bra \alpha, \beta\ket}\bra \alpha, \beta\ket \hat{e}_{\alpha + \beta}.
\end{equation*}

\noindent Fix an element $\sigma$ of the maximal ideal of $\widehat{\g}$. The natural $q$-deformation of $\exp(\operatorname{Li}_2(\sigma e_{\alpha}))$ is given by the $q$-dilogarithm,
\begin{equation}
{\bf E}(\sigma\hat{e}_{\alpha}) = \exp\left(-\sum_{j \geq 1}\frac{\sigma^j\hat{e}_{j\alpha}}{j((-q^{\frac{1}{2}})^j - (-q^{\frac{1}{2}})^{-j})}\right).
\end{equation}
Fix $n \in \Z$ and $\sigma'$ in the maximal ideal. We consider the adjoint action on $\widehat{\g}_q$ of (shifts and powers of) $q$-dilogarithms, 
\begin{equation*}
{\bf U}^{\Omega}((-q^{\frac{1}{2}})^n\sigma'\hat{e}_{\alpha}) = \Ad({\bf E}^{\Omega}((-q^{\frac{1}{2}})^n\sigma'\hat{e}_{\alpha})). 
\end{equation*}
Using the identity
\begin{equation*} 
{\bf E}(\sigma\hat{e}_{\alpha}) = \prod_{k \geq 0} \left(1 + q^{k + \frac{1}{2}}\sigma\hat{e}_{\alpha}\right)^{-1},
\end{equation*}
one can check that the adjoint action is given by
\begin{equation}
{\bf U}((-q^{\frac{1}{2}})^n\sigma'\hat{e}_{\alpha})(\hat{e}_{\beta}) = \hat{e}_{\beta} \prod^{\bra \alpha, \beta \ket - 1}_{j = 0} \left(1 + (-1)^n q^{j + \frac{n+1}{2}} \sigma'\hat{e}_{\alpha}\right) \prod^{ - 1}_{j = \bra \alpha, \beta \ket} \left(1 + (-1)^n q^{j + \frac{n+1}{2}} \sigma'\hat{e}_{\alpha}\right)^{-1}.
\end{equation}
(following the convention that the empty product equals $1$).\\

\noindent As usual we suppose from now that $\Gamma$ is generated by elements ${\gamma_1, \dots, \gamma_{\ell_1}}$ and $\eta_1, \dots \eta_{\ell_2}$ such that $\bra \gamma_i, \gamma_j\ket = \bra \eta_i, \eta_j \ket = 0, \bra \gamma_i, \eta_j\ket = 1$. As in the classical case, we set $\widehat{\g}_q = \g_q \otimes_{\C} R_k$.\\ 

\noindent A \emph{refined BPS spectrum} (for the fixed central charge $Z$) is a set of functions $\Omega_n\!: \Gamma \to \mathbb{Z}$ such that $\Omega_n(\gamma) = \Omega_n(-\gamma)$ and $\Omega_n(0) = 0$, for $n \in \Z$. Refined BPS rays are defined in the obvious way.\\ 

\noindent Suppose we have a collection of elements $a_{\gamma'} \in \widehat{\g}_q$, labelled by $\gamma' \in \Gamma$. Then we will denote by $\prod^{\z} a_{\gamma'}$ the product of the $a_{\gamma'}$, taken in the clockwise order of $Z(\gamma')$, starting from $\z \in \C^*$. We introduce a $q$-deformation of the operator \eqref{TBop}, for the same values of $\z$, acting on a suitable holomorphic family $\widehat{\psi}(\z)$ of elements of $\operatorname{GL}(\widehat{\g}_q)$ as
\begin{align*}
\nonumber \widehat{\Phi}(\widehat{\psi}(\z))_{\alpha} &= \widehat{\psi}^{0}_{\alpha}(\z) \times \\ 
\nonumber &\prod_{\gamma'}^{\z}\exp\left(\sum_n (-1)^n\Omega_n(\gamma') \sum^{\bra \gamma', \alpha \ket - 1}_{j = 0}\int_{\ell_{\gamma'}}\frac{d\z'}{\z'}\rho(\z, \z')\log \left(1 + q^{j + \frac{n+1}{2}}(s, t)^{\gamma'}\widehat{\psi}_{\gamma'}(\z')\right) \right.\\
&\quad\quad\left. -\sum_{n}(-1)^n\Omega_n(\gamma')\sum^{ - 1}_{j = \bra \gamma', \alpha \ket} \int_{\ell_{\gamma'}}\frac{d\z'}{\z'}\rho(\z, \z')\log\left(1 + q^{j + \frac{n+1}{2}} (s, t)^{\gamma'}\widehat{\psi}_{\gamma'}(\z')\right)\right).
\end{align*}
(extended by linearity). Here $\widehat{\psi}^{0}_{\alpha}(\z) = \exp(\pi R (\z^{-1} Z(\alpha) + \z\bar{Z}(\alpha))) \hat{e}_{\alpha}$. The relevant TBA type equation is
\begin{equation}\label{qTBA}
\widehat{\Phi}(\psi) = \psi. 
\end{equation}
\begin{rmk} Formally \eqref{qTBA} is very similar to \eqref{TBA}, but notice that even when $\psi$ takes values in $\Aut_{R_k}(\widehat{\g}_q)$, in general $\widehat{\Phi}(\psi)$ is only an algebra automorphism to first order (i.e. ignoring higher order brackets in the Baker-Campbell-Hausdorff formula). Moreover it seems that the presence of the operator $\prod^{\z}$ makes this integral equation rather more complicated (in particular, it is much more nonlinear).
\end{rmk}

\noindent As in Lemma \ref{RH}, one proves that a solution of \eqref{qTBA} solves the Riemann-Hilbert problem with rays $\ell_{\alpha}$ and Stokes factors $\prod_n {\bf U}^{(-1)^n\Omega_n(\alpha)}((-q^{\frac{1}{2}})^n(s, t)^{\alpha}\hat{e}_{\alpha})$.\\ 

\noindent Fix the same family of central charges $Z^\tau$ as in the classical case. We prescribe a refined BPS spectrum for $\tau < \tau_0$ by $\Omega_0(\pm\gamma_i) = \Omega_0(\pm\eta_j)=1$ for all $i, j$, while all other $\Omega_n$ vanish. We will only be concerned with \eqref{qTBA} in this special case. That is, we only look at the component
\begin{align}\label{qTBAspecial}
\nonumber\widehat{\psi}_{m\gamma_i}(\z) = \widehat{\psi}^0_{m \gamma_i}(\z)&\exp\left(-\sum_j \sum^{ - 1}_{l = \bra \eta_j, m\gamma_i \ket} \int_{\ell_{\eta_j}}\frac{d\z'}{\z'}\rho(\z, \z')\log\left(1 + q^{l + \frac{1}{2}} \widehat{\psi}_{\eta_j}(\z')\right) \right)\\
&\exp\left(\sum^{ \bra -\eta_j, m\gamma_i \ket}_{l = 0} \int_{\ell_{-\eta_j}}\frac{d\z'}{\z'}\rho(\z, \z')\log\left(1 + q^{l + \frac{1}{2}} \widehat{\psi}_{-\eta_j}(\z')\right)\right),
\end{align}
and the corresponding one for $\widehat{\psi}_{n\eta_j}(\z)$. We expand the argument in the first exponential in \eqref{qTBAspecial} as
\begin{equation*}
\sum_j \sum_{h \geq 1} \int_{\ell_{\eta_j}} \frac{d\z'}{\z'}\rho(\z, \z')\lambda^m_h (q)\widehat{\psi}_{h \eta_j}(\z'),
\end{equation*}
where
\begin{equation*}
\lambda^m_h = -\frac{(-1)^h}{h}\sum^{-1}_{l = -m} q^{(l + \frac{1}{2}) h}.
\end{equation*}
The corresponding quantity for the $\widehat{\psi}_{m\eta_j}(\z)$ component is given by
\begin{equation*}
\mu^n_h = \frac{(-1)^h}{h}\sum^{n-1}_{l = 0} q^{(l + \frac{1}{2}) h}.
\end{equation*}
Following the argument leading to \eqref{iterativeSol} we see that there is an expansion
\begin{equation*} 
\widehat{\psi}^{\infty}_{\eta_j}(\z) = \widehat{\psi}^0_{\eta_j}(\z)\exp \sum_p \mu^1_p \sum_{\gamma_{\T} = p\gamma_i} \widehat{\W}_{\T} \widehat{\G}_{\T}(\z)
\end{equation*}
for some $\widehat{\W}_{\T} \in \C[q^{\pm \frac{1}{2}}]$ and where $\widehat{\G}_{T}(\z)$ is defined recursively by 
\begin{equation*}
\widehat{\G}_{\T}(\z) = \int_{\ell_{\gamma_{\T}}} \frac{d\z'}{\z'}\rho(\z, \z') \widehat{\psi}^0_{\gamma_{\T}}(\z')\prod_{\T'}\widehat{\G}_{\T'}(\z'),
\end{equation*}
as in \eqref{propagators}. Thus $\lambda^m_h$ is the factor of $\widehat{\W}_{\T}$ associated to a directed edge $\{m\gamma_i\} \to \{n\eta_j\}$ in $\T$. Similarly $\mu^n_h$ is the factor attached to an edge $\{n\eta_j\} \to \{h\gamma_i\}$. Restricting to $\T$ whose root is some $m\gamma_i$ we find
\begin{equation*}
\widehat{\W}_{\T} = \prod_{E\in\T^{[1]}} \lambda(E) \mu(E).
\end{equation*}
We can go through the proof of Lemma \ref{LemmaA}, replacing $\G_{\T}(\z)$ with $\widehat{\G}_{\T}(\z)$. Then we see that the same set of tropical types $\{[\Upsilon_j]\}_{\T}$ emerges, but now the key point is that these types appear naturally weighted by a monomial $m([\Upsilon_j]) \in \C[q^{\pm\frac{1}{2}}]$. To see this notice that in the residue terms \eqref{res2}, \eqref{res3} the factors $(-1)^{\alpha(v) + \alpha(v')}\psi^0_{\alpha(v) + \alpha(v')}$, respectively $(-1)^{\alpha(v) + \alpha(v''_k)}\psi^0_{\alpha(v) + \alpha(v''_k)}$ must be replaced with
\begin{equation*}
\widehat{\psi}^0_{\alpha(v)}\widehat{\psi}^0_{\alpha(v')} = q^{\frac{1}{2}\bra\alpha(v),\alpha(v')\ket}\widehat{\psi}^0_{\alpha(v) + \alpha(v')}, \quad \widehat{\psi}^0_{\alpha(v)}\widehat{\psi}^0_{\alpha(v''_k)} = q^{\frac{1}{2}\bra\alpha(v),\alpha(v''_k)\ket}\widehat{\psi}^0_{\alpha(v) + \alpha(v''_k)}.
\end{equation*}
Let
\begin{equation*}
\widehat{J}_{\T}(\eps, R)\hat{e}_{c(\T)} = \widehat{\G}_{\T}(\z, \tau_0 - \eps) - \G_{\T}(\z, \tau_0 + \eps).
\end{equation*} 
We find a leading term for $\widehat{J}_{\T}(\eps, R)$ given by 
\begin{equation*}
\sum_{\{[\Upsilon_j]\}_{\T}} m([\Upsilon_j]) f(\z, R).
\end{equation*}  
We take the sum over instanton contributions for trees $\T$ labelled by two charges $\gamma, \eta$, of fixed tropical degree ${\bf w}$, and with $\gamma_{\T} = w_{11}\gamma$.  Then following the proof of Theorem \ref{TheoremB}, but replacing the relevant GPS theory with its $q$-deformation as in \cite{us} section 4, one can prove
\begin{equation*}
\sum_{\deg{\T} = {\bf w}, \gamma_{\T} = w_{11}\gamma} \widehat{\W}_{\T} \left(\sum_{\{[\Upsilon_j]\}_{\T}} m([\Upsilon_j])\right) = \prod_{i, j}\frac{1}{w_{i j}[w_{i j}]_q}\frac{\widehat{N}^{\rm trop}({\bf w})}{\Aut({\bf w}')}, 
\end{equation*}
where the $q$-deformed tropical counts $\widehat{N}^{\rm trop}({\bf w})$ are a special case of the Block-G\"ottsche invariants (\cite{gottsche}, \cite{mik}), as discussed in \cite{us} section 4.\\ 

\begin{exm} The $q$-deformed correction
\begin{equation*} 
\int_{\ell^+_{\gamma}}\frac{d\z_1}{\z_1}\rho(\z, \z_1)\widehat{\psi}^0_{\gamma}(\z_1)\int_{\ell^+_{\eta}}\frac{d\z_2}{\z_2}\rho(\z_1, \z_2)\widehat{\psi}^0_{\eta}(\z_2)
\end{equation*}
leads to a residue term 
\begin{equation*} 
\int_{\ell^-_{\gamma + \eta}}\frac{d\z'}{\z'}\rho(\z, \z')q^{\frac{1}{2}}\widehat{\psi}^0_{\gamma + \eta}(\z'),
\end{equation*}
so for $[\Upsilon]$ the type of a tropical line,
\begin{equation*}
m([\Upsilon]) = q^{\frac{1}{2}}, \quad \widehat{N}^{\rm trop}(1,1) = \lambda^1_1 q^{\frac{1}{2}} = 1.
\end{equation*}
\end{exm}

\vspace{.5cm}
\noindent Dipartimento di Matematica ``F. Casorati"\\
Universit\`a di Pavia, via Ferrata 1, 27100 Pavia, Italia\\

\noindent
{\tt{saraangela.filippini@unipv.it}\\
\tt{jacopo.stoppa@unipv.it}}

\end{document}